\documentclass[final,3p,times,12pt]{elsarticle}
\usepackage{amsmath}
\usepackage{amsfonts}
\usepackage{amssymb}
\usepackage{graphicx}
\usepackage{subfigure}
\usepackage{color}
\usepackage{array}
\usepackage{mathrsfs}
\usepackage{float}
\usepackage{epstopdf}
\usepackage{booktabs}
\usepackage{tabularx}
\usepackage{lineno,hyperref}

\newtheorem{theorem}{Theorem}[section]
\newtheorem{proposition}{Proposition}[section]
\newtheorem{lemma}{Lemma}[section]
\newtheorem{definition}{Definition}[section]

\newtheorem{corollary}{Corollary}[section]
\newenvironment{proof}[1]{{\noindent {\emph{Proof}}}\quad}{\hfill $\square$\par}
\newenvironment{proofoftheorem}[1]{{\noindent\it Proof of theorem #1}\quad}{\hfill $\square$\par}

\begin{document}
\title{Littlewood-Paley Characterization for Musielak-Orlicz-Hardy Spaces Associated with Operators}
\author{Jiawei Shen$^1$ $\cdot$ Zhitian Chen$^1$ $\cdot$ ShunchaoLong$^1$}
\maketitle
\newcommand\blfootnote[1]{%
\begingroup
\renewcommand\thefootnote{}\footnote{#1}%
\addtocounter{footnote}{-1}%
\endgroup
}

\blfootnote{Shunchao Long: sclong@xtu.edu.cn}
\footnotetext[1]{School of Mathematics and Computational Science, Xiangtan University, Xiangtan, 411105, China}

\noindent \textbf{Abstract}\quad Let $X$ be a space of homogeneous type. Assume that $L$ is an non-negative second-order self-adjoint operator on $L^2\left(X\right)$ with (heart) kernel associated to the semigroup $e^{ - tL}$ that satisfies the Gaussian upper bound. In this paper, the authors introduce a new characterization of the Musielak-Orlicz-Hardy Space $H_{\varphi, L}\left(X\right)$ associated with $L$ in terms of the Lusin area function where $\varphi$ is a growth function. Further, the authors prove that the Musielak-Orlicz-Hardy Space $H_{L,G,\varphi}\left(X\right)$ associated with $L$ in terms of the Littlewood-Paley function is coincide with $H_{\varphi, L}\left(X\right)$ and their norms are equivalent.

\quad


\quad

\section{Introduction}
Recently, the study of the Hardy spaces associated with operators has been in the spotlight. This topic was initiated by \citet{auscher:2005}, who studied the Hardy space $H^1_{L}\left(\mathbb{R}^n\right)$ associated with the operator $L$ whose heat kernel satisfies the pointwise Poisson upper bounded condition. Later on, the adapted BMO theory has been presented by \citet{Duong:2005,Thinh:2010}, under the assumption that the heat kernel associated to $L$ satisfies the pointwise Gaussian estimate. The theory of the Hardy space $H^p_{L}\left(\mathbb{R}^n\right)$ for $0 < p < 1$ associated with the operator $L$ satisfying the Davies-Gaffney estimates was established by \citet{Yan:2008}.
It is then quite natural to consider the weighted Hardy spaces $H^p_{L,\omega}$ associated with an operator $L$ and a weight function $\omega$. \citet{Song:2010} first introduced the weighted Hardy space $H^1_{L,\omega}\left(\mathbb{R}^n\right)$ associated with the Schrödinger operator $L$ for $\omega \in A_{\infty}\left(\mathbb{R}^n\right)$. Recently, \citet{Duong:2016} considered two kinds of weighted Hardy spaces on the homogeneous spaces $X$ associated with an operator whose kernel satisfying the Gaussian upper bound. For $0 < p \le 1$ and $\omega \in A_{\infty}$, they first studied the weighted Hardy space $H^p_{L, S, \omega}\left(X\right)$ which defined in terms of the Lusin area function, and secondly turned to consider the weighted Hardy space $H^p_{L, G, \omega}\left(X\right)$ which defined in terms of the Littlewood-Paley function. Finally, they obtained the equivalence between the two kinds of weighted Hardy spaces by adding the Moser type condition for $L$. Subsequently, the equivalence of these two kinds of weighted Hardy spaces was demonstrated by \citet{Guorong:2016} without using the additional Moser type estimate.

On the other hannd, \citet{Ky:2011} presented a new Musielak-Orlicz-Hardy space, $H_{\varphi}\left(\mathbb{R}^n\right)$, defined via a growth function $\varphi$ (see Sect.2 below for the definition of growth function). As an natural generalization, the Musielak-Orlicz-Hardy space $H_{\varphi, L}$, defined via the Lusin area function associated with an operator $L$ that satisfies the Davies-Gaffney estimate, which contains the weighted Hardy space $H^p_{L, S, w}\left(X\right)$ in \citep{Duong:2016}, had been introduced and systematically studied by Yang et al. in \cite{Yang:2012} later on. Characterizations of $H_{\varphi, L}$, including the atom, the molecule, etc. was obtained in \cite{Yang:2012}. However, to characterize $H_{\varphi, L}$, Yang et al. needed to impose an extra assumption that the growth function $\varphi$ satisfies the uniformly reverse H\"older condition.

Throughout this article $\left(X, d, \mu\right)$ is a metric measure space endowed with a distance $d$ and a non-negative Borel doubling measure $\mu$. And we assume that $L$ is a densely defined operator on $L^2 \left(X\right)$ and satisfies the following two conditions in different sections of this paper.
{
\begin{enumerate}[(\textbf{H}1)]
  \renewcommand{\labelenumi}{(\textbf{H\theenumi})}
  \setlength{\leftmargin}{1.2em}
  \item $L$ is a second-order non-negative self-adjoint operator on $L^2 \left(X\right)$;
  \item The kernel of $e^{-tL}$, denote by $p_t \left(x, y\right)$, is a measurable function on $X \times X$ and satisfies the Gaussian estimates, namely, there exist positive constants $C_{1}$ and $C_{2}$ such that, for all $t > 0$, and $x, y \in X$,
      {
      \setlength\abovedisplayskip{1ex}
      \setlength\belowdisplayskip{1ex}
      \begin{equation}\label{Eq.GE}
        \left| {{p_t}\left( {x,y} \right)} \right| \le \frac{{{C_1}}}{{V\left( {x,\sqrt t } \right)}}\exp \left( { - \frac{{d{{\left( {x,y} \right)}^2}}}{{{C_2}t}}} \right),
      \end{equation}
      where $V\left(x, \sqrt{t}\right) = \mu \left(B\left(x, \sqrt{t}\right)\right)$.
      }
\end{enumerate}
}
Given an operator $L$ that satisfying (\textbf{H1}) and (\textbf{H2}) and a function $f \in L^2 \left(X\right)$, we consider the following Littlewood-Paley function $G_{L}\left(f\right)$ and Lusin area function $S_{L}\left(f\right)$ associated with the heat semigroup generated by $L$
{
\setlength\abovedisplayskip{1ex}
\setlength\belowdisplayskip{1ex}
\begin{equation}\label{Eq.Littlewood_Paley_Function}
  {G_L}\left( f \right)\left( x \right): = {\left( {\int_0^\infty  {|{t^2}L{e^{ - {t^2}L}}f\left( x \right){|^2}\frac{{dt}}{t}} } \right)^{{1 \mathord{\left/
 {\vphantom {1 2}} \right.
 \kern-\nulldelimiterspace} 2}}}
\end{equation}
and
\begin{equation}\label{Eq.Lusin_Area_Function}
  {S_L}\left( f \right)\left( x \right): = {\left( {\int_0^\infty  {\int_{d\left( {x,y} \right) < t} {|{t^2}L{e^{ - {t^2}L}}f\left( y \right){|^2}\frac{{d\mu \left( y \right)}}{{\mu \left( {B\left( {x,t} \right)} \right)}}} \frac{{dt}}{t}} } \right)^{{1 \mathord{\left/
 {\vphantom {1 2}} \right.
 \kern-\nulldelimiterspace} 2}}}.
\end{equation}
}

In this paper, Musielark-Orlicz Hardy spaces $H_{\varphi, L}$ and $H_{L, G, \varphi}$ will be concerned. Their definitions are as follows.
{
  \setlength\abovedisplayskip{1ex}  
  \setlength\belowdisplayskip{1ex}
\begin{definition}\label{Def.H_S_L}
Let $L$ satisfies \rm{(\textbf{H1})} and \rm{(\textbf{H2})}, \emph{and $\varphi$ be a growth function. A function $f\in H^{2}(X)$ is said to be in ${\tilde H_{\varphi ,L}}( X )$ if $S_{L}\left(f\right) \in L^\varphi \left(X\right)$. Moreover, define
$$
  \setlength\abovedisplayskip{1ex}  
  \setlength\belowdisplayskip{1ex}
  {\left\| f \right\|_{{H_{\varphi ,L}}\left( X \right)}}: = {\left\| {{S_L}\left( f \right)} \right\|_{{L^\varphi }\left( X \right)}}: = \inf \left\{ {\lambda  \in \left( {0,\infty } \right);\int_X {\varphi \left( {x,\frac{{{S_L}\left( f \right)(x)}}{\lambda }} \right)d\mu \left( x \right)}  \le 1} \right\}.
$$
The Musielak-Orlicz-Hardy space $H_{\varphi, L}\left(X\right)$ is defined to be the completion of ${{\tilde H}_{\varphi ,L}}\left( X \right)$ with the quasi-norm ${\left\|  \cdot  \right\|_{{H_{\varphi ,L}}\left( X \right)}}$}.
\end{definition}
\begin{definition}\label{Def.H_G_L}
Let $L$ satisfies \rm{(\textbf{H1})} and \rm{(\textbf{H2})}, \emph{and $\varphi$ be a growth function. A function $f\in H^{2}(X)$ is said to be in ${\tilde H_{L, G, \varphi}}( X )$ if $G_{L}\left(f\right) \in L^\varphi \left(X\right)$. Moreover, define
$$
  \setlength\abovedisplayskip{1ex}  
  \setlength\belowdisplayskip{1ex}
  {\left\| f \right\|_{{H_{L, G, \varphi}}\left( X \right)}}: = {\left\| {{G_L}\left( f \right)} \right\|_{{L^\varphi }\left( X \right)}}: = \inf \left\{ {\lambda  \in \left( {0,\infty } \right);\int_X {\varphi \left( {x,\frac{{{G_L}\left( f \right)(x)}}{\lambda }} \right)d\mu \left( x \right)}  \le 1} \right\}.
$$
The Musielak-Orlicz-Hardy space $H_{L, G, \varphi}\left(X\right)$ is defined to be the completion of ${{\tilde H}_{L, G, \varphi}}\left( X \right)$ with the quasi-norm ${\left\|  \cdot  \right\|_{{H_{L, G, \varphi}}\left( X \right)}}$}.
\end{definition}
}

What deserves to be mentioned the most is that the Musielak-Orlicz Hardy space $H_{\varphi, L}$ introduced in \citep{Yang:2012} is associated with $L$ satisfying the Davies-Gaffney estimates, while the operator $L$ in definition \ref{Def.H_S_L} and definition \ref{Def.H_G_L} satisfies the stronger Gaussian estimates.

Motivated by the work of \citep{Duong:2016, Yang:2012, Guorong:2016}, the first contribution of this paper is to establish a discrete characterization for the two kinds of Musielak-Orlicz-Hardy spaces $H_{\varphi, L}\left(X\right)$ and $H_{L, G, \varphi}\left(X\right)$ defined above.
This generalizes the results presented in \citep{Duong:2016, Guorong:2016} since $H_{\varphi, L}\left(X\right)$ contains $H^p_{L, S, w}\left(X\right)$ and $H_{L, G, \varphi}\left(X\right)$ contains $H^p_{L, G, w}\left(X\right)$.
Also, by removing the uniformly reverse H\"older condition on growth function $\varphi$, our work improves a part of results of \citet{Yang:2012}.
The second goal of this article is to prove that $H_{\varphi, L}\left(X\right)$ and $H_{L, G, \varphi}\left(X\right)$ are equivalent, which improves the result about the behavior of Littlewood-Paley $g$-function $G_L$ on $H_{\varphi, L}$ proved in \citep[Theorem 6.3]{Yang:2012}.

Our main approach is inspired by the results in \citep{Weiss:1991, Duong:2016}.The layout of this article is as follows. We first recall some basic facts and known results in Sect. 2. In Sect.3, we first establish discrete characterizations for $H_{\varphi, L}\left(X\right)$ and $H_{L, G, \varphi}\left(X\right)$ and then obtain the consistency between $H_{\varphi, L}\left(X\right)$ and $H_{L, G, \varphi}\left(X\right)$ in the sense of norm as a corollary.

Throughout this paper, we mean by writting $a \cong b$ that variables $a$ and $b$ are equivalent, namely, there exist positive constants $C_1$ and $C_2$ independent of $a$ and $b$ such that $C_1 b \le a \le C_2 b$.

\section{Preliminaries}
\subsection{Metric Measure Spaces}
Let $\left(X, d, \mu \right)$ be a metric measure space, namely, $d$ is a metric and $\mu$ a nonnegative Borel regular measure on $X$. Throughout out this paper, for any fixed $x \in X$ and $r \in \left(0, \infty \right)$, we denote the open ball centered at $x$ with radius $r$ by
$$
\setlength\abovedisplayskip{1ex}
\setlength\belowdisplayskip{1ex}
B\left( {x,r} \right): = \left\{ {y \in X;d\left( {x,y} \right) < r} \right\},
$$
and we set $V\left(x, r\right):= \mu \left( B\left(x, r\right)\right)$. Moreover, we assume that $X$ is of homogeneous type, that is, there exists a constant $C_D \in \left[1, \infty\right)$ such that, for all $x \in X$ and $r \in \left(0, \infty\right)$,
{
\setlength\abovedisplayskip{1ex}
\setlength\belowdisplayskip{1ex}
\begin{equation}\label{Eq.Property_of_Homogeneous_Type}
  V\left( {x,2r} \right) \le C_D V\left( {x,r} \right) < \infty.
\end{equation}

Condition \eqref{Eq.Property_of_Homogeneous_Type} is also called the doubling condition which implies that the following strong homogeneity property that, for some positive constants $C$ and $n$,
}
{
\setlength\abovedisplayskip{1ex}
\setlength\belowdisplayskip{1ex}
\begin{equation}\label{Eq.Doubling_dim_of_X_01}
  V\left( x, \lambda r \right) \le C \lambda ^n V\left( x, r \right)
\end{equation}
uniformly for all $\lambda \in \left[1, \infty\right)$, $x \in X$, and $r \in \left(0, \infty\right)$. And as is shown by \citet{Grigor:2009}, let $C_D$ be as in \eqref{Eq.Property_of_Homogeneous_Type} and $m = {\log _2}{C_D}$, then for all $x, y\in X$ and $0 < r \le R < \infty$ we have
}
{
\setlength\abovedisplayskip{0.5ex}
\setlength\belowdisplayskip{1ex}
\begin{equation}\label{Eq.Doubling_dim_of_X_02}
  V\left( {x,R} \right) \le {C_D}{\left[ {\frac{{R + d\left( {x,y} \right)}}{r}} \right]^m}V\left( {y,r} \right).
\end{equation}
Using the doubling condition \eqref{Eq.Property_of_Homogeneous_Type}, it is trivial to show that for any $N > n$, there exists a constant $C_N$ such that for all $x \in X$ and $t > 0$,
}
{
\setlength\abovedisplayskip{1ex}
\setlength\belowdisplayskip{1ex}
\begin{equation}\label{Eq.Doubling_Int_Properties}
\int_X {{{\left( {1 + {t^{ - 1}}d\left( {x,y} \right)} \right)}^{ - N}}d\mu \left( y \right)}  \le {C_N}V\left( {x,t} \right).
\end{equation}

We further have the following dyadic cubes decomposition on spaces of homogeneous type constructed by \citet{Christ:1990}.
}
{
\setlength\abovedisplayskip{0.5ex}
\setlength\belowdisplayskip{1ex}
\begin{lemma}{}\label{Lemma_Decompositon_of_Space}
    Let $\left( {X,d,\mu } \right)$ be a space of homogeneous type. Then, there exist a collection\\ $\left\{ {Q_\alpha ^k \subset X:k \in \mathbb{Z},\alpha  \in {I_k}} \right\}$ of open subsets, where $I_{k}$ is some index set, and constants $\delta  \in \left( {0,1} \right)$, and $C_{1}, C_{2} > 0$, such that
    \begin{enumerate}[\rm(i)]
        \item $\mu \left( {X\backslash { \cup _\alpha }Q_\alpha ^k} \right) = 0$, for each fixed $k$ and $Q_\alpha ^k \cap Q_\beta ^k = \emptyset $ if $\alpha  \ne \beta $;
        \item for any $\alpha, \beta, k, l$ with $k \le l$, either $Q_\beta ^l \subset Q_\alpha ^k$ or $Q_\beta ^l \cap Q_\alpha ^k = \emptyset $;
        \item for each $\left( k, \alpha\right)$ and each $l < k$, there exists a unique $\beta \in I_{l}$ such that $Q_\alpha ^k \subset Q_\beta ^l$;
        \item ${\rm{diam}}\left( {Q_\alpha ^k} \right) \le {C_1}{\delta ^k}$;
        \item each ${Q_\alpha ^k}$ contains some ball $B\left( {z_\alpha ^k,{C_2}{\delta ^k}} \right)$, where $z_\alpha ^k \in X$.
    \end{enumerate}
\end{lemma}

We can think of $Q_{\alpha}^k$ as being a dyadic cube with diameter roughly $\delta ^k$ centered at $y_{Q_{\alpha}^k}$, and we then set $\ell ( {Q_\alpha ^k} ) = {C_1}{\delta ^k}$. The precise value $C_1$ is nonessential, and as was proved by \citet{Christ:1990}, in what follows, we without loss of generality assume $C_1 = \delta ^{-1}$.
}

\subsection{Growth Functions and Their Properties}
 Recall from \citep{Yang:2012} that a nonnegative nondecreasing function $\varPhi$ defined on $\left[ {0, + \infty } \right)$ is said to be an \emph{Orlicz function} if
$\varPhi\left(0\right) = 0$, $\varPhi\left(t\right) > 0$ for all $t \in \left(0, \infty \right)$ and ${\lim _{t \to \infty }}\varPhi \left( t \right) = \infty $. The function $\varPhi$ is said to be of \emph{upper type} $p$ (resp., \emph{lower type} $p$)
for some $p \in \left[ 0, \infty\right)$, if there exists a positive constant $C$ such that, for all $t \in \left[ 1, \infty\right)$ (resp., $t \in \left[0, 1\right]$) and $s \in \left[0, \infty \right)$, $\varPhi\left(st\right) \le C t^{p} \varPhi\left(s\right)$.
And it is trivial that an Orlicz function is of upper type $1$, if it is of upper type $p \in \left(0, 1\right)$.

Let $\varphi :X \times \left[ {0, + \infty } \right) \to \left[ {0, + \infty } \right)$ be a function, such that, for any $x \in X$, $\varphi\left(x, \cdot \right)$ is an Orlicz function. We say that $\varphi$ is of \emph{uniformly upper type} $p$ (resp., \emph{uniformly lower type} $p$) for some $p \in \left[0, \infty \right)$, if there exists a positive constant $C$ such that, for all $x \in X$, $t \in \left[1, \infty\right)$ (resp., $t \in \left[0, 1\right]$) and $s \in \left[0, \infty \right)$,
{
\setlength\abovedisplayskip{1ex}
\setlength\belowdisplayskip{1ex}
\begin{equation}\label{Eq.Property_of_Growth_Function}
  \varphi \left( {x,st} \right) \le C{t^p}\varphi \left( {x,s} \right).
\end{equation}

As in \citet{Ky:2011}, a function $\varphi :X \times \left[ {0, + \infty } \right) \to \left[ {0, + \infty } \right)$ is said to be \emph{uniformly locally integrable}, if for all $t \in \left[0, \infty\right)$, $x \mapsto \varphi \left( {x,t} \right)$ is measurable and for all bounded subsets $K \subset X$,
}
$$
\setlength\abovedisplayskip{1ex}
\setlength\belowdisplayskip{0ex}
\int_K {\mathop {\sup }\limits_{t \in (0,\infty )} \left\{ {\varphi (x,t){{\left[ {\int_K {\varphi (y,t)d\mu (y)} } \right]}^{ - 1}}} \right\}d\mu (x)} < \infty.
$$

Following \citep{Yang:2012, Yang:2017}, we next recall the definition of the \emph{Uniform Muchenhoupt Class} and its properties.
{
\setlength\abovedisplayskip{0ex}
\setlength\belowdisplayskip{1ex}
\begin{definition}\label{Def.A_p_RH_p}\rm{
  Let $\varphi :X \times \left[ {0, + \infty } \right) \to \left[ {0, + \infty } \right)$ be uniformly locally integrable.
    The function $\varphi\left(\cdot, t\right)$ is said to satisfy the \emph{uniformly Muckenhoupt condition} for some $q \in \left[1, \infty\right)$, denoted by $\varphi \in \mathbb{A}_{q}\left(X \right)$, if, when $q \in \left(1, \infty\right)$,
        {
        \setlength\abovedisplayskip{1ex}
        \setlength\belowdisplayskip{1ex}
        \begin{multline*}
        {\mathbb{A}_q}\left( \varphi  \right): = \mathop {\sup }\limits_{t \in (0,\infty )} \mathop {\sup }\limits_{B \subset X} \left\{ {\frac{1}{{\mu \left( B \right)}}\int_B {\varphi \left( {x,t} \right)d\mu (x)} } \right\} \\
        \times {\left\{ {\frac{1}{{\mu \left( B \right)}}\int_B {{{\left[ {\varphi \left( {y,t} \right)} \right]}^{{{ - q'} \mathord{\left/
        {\vphantom {{ - q'} q}} \right.
        \kern-\nulldelimiterspace} q}}}d\mu (y)} } \right\}^{{q \mathord{\left/
        {\vphantom {q {q'}}} \right.
        \kern-\nulldelimiterspace} {q'}}}} < \infty,
        \end{multline*}
        }
        where ${1 \mathord{\left/
        {\vphantom {1 q}} \right.
        \kern-\nulldelimiterspace} q} + {1 \mathord{\left/
        {\vphantom {1 {q'}}} \right.
        \kern-\nulldelimiterspace} {q'}} = 1$, or
        $$
        \setlength\abovedisplayskip{1ex}
        \setlength\belowdisplayskip{1ex}
        {\mathbb{A}_1}\left( \varphi  \right): = \mathop {\sup }\limits_{t \in (0,\infty )} \mathop {\sup }\limits_{B \subset X} \frac{1}{{\mu \left( B \right)}}\int_B {\varphi \left( {x,t} \right)d\mu (x)} \left\{ {\mathop {{\rm{esssup}}}\limits_{y \in B} {{\left[ {\varphi (y,t)} \right]}^{ - 1}}} \right\} < \infty.
        $$
        We further define $\mathbb{A}_{\infty}\left(X\right): = { \bigcup _{q \in \left[ {1,\infty } \right)}}{\mathbb{A}_q}\left( X \right)$ and let
        $$
        \setlength\abovedisplayskip{1ex}
        \setlength\belowdisplayskip{1ex}
        q\left( \varphi  \right): = \inf \left\{ {q \in \left[ {1,\infty } \right);\varphi  \in {\mathbb{A}_q}\left( X \right)} \right\}
        $$
        to be the \emph{critical indices} of $\varphi$.
  }
\end{definition}

The following properties of $\mathbb{A}_\infty\left(X\right)$ and their proofs are similar to those
in \citet{Yang:2017}, and we omit the details. In what follows, we use the notation
$$
\setlength\abovedisplayskip{1ex}
\setlength\belowdisplayskip{1ex}
\varphi \left( {E,t} \right): = \int_E {\varphi \left( {x,t} \right)d\mu \left( x \right)}
$$
for any measurable subset $E$ of $X$ and $t \in \left[0, \infty\right)$. And $\mathcal{M}$ denotes
the Hardy-Liitlewood maximal function on $X$, i.e., for all $x\in X$,
$$
\setlength\abovedisplayskip{1ex}
\setlength\belowdisplayskip{1ex}
\mathcal{M}\left( f \right)\left( x\right): = \mathop {\sup }\limits_{B \ni x} \frac{1}{{\mu\left( B \right)}}\int_B {\left| {f\left( y \right)} \right|d\mu \left( y \right)},
$$
where the supremum is taken over all balls $B \ni x$.
\begin{lemma}{}\label{Lemma_Properties_of_Growth_Functions_03}
\begin{enumerate}

  \item ${\mathbb{A}_1}\left( X \right) \subset {\mathbb{A}_p}\left( X \right) \subset {\mathbb{A}_q}\left( X \right)$, for $1 \le p \le q < \infty $.
  \item If $\varphi \in \mathbb{A}_{p}\left(X\right)$ with $p\in\left(1, \infty\right)$, then there exists some $q\in \left(1, p\right)$ such that $\varphi \in \mathbb{A}_{q}\left(X\right)$.
  \item If $\varphi \in \mathbb{A}_{p}\left(X\right)$ with $p\in\left(1, \infty\right)$, then there exists a positive constant $C$ such that, for all measurable functions $f$ on $X$ and $t\in\left[0, \infty\right)$,
      $$
      \setlength\abovedisplayskip{1ex}
      \setlength\belowdisplayskip{1ex}
      \int_X {{{\left[ {\mathcal{M}\left( f \right)\left( x \right)} \right]}^p}\varphi \left( {x,t} \right)d\mu \left( x \right)}  \le C\int_X {{{\left| {f\left( x \right)} \right|}^p}\varphi \left( {x,t} \right)d\mu \left( x \right)} .
      $$
  \item If $\varphi \in \mathbb{A}_{p}\left(X\right)$ with $p\in\left[1, \infty\right)$, then there exists a positive constant $C$ such that, for all balls $B \subset X$ and measurable subset $E \subset B$ and $t\in \left[0, \infty\right)$, $\frac{{\varphi \left( {{B},t} \right)}}{{\varphi \left( {{E},t} \right)}} \le C{\left[ {\frac{{\mu \left( {{B}} \right)}}{{\mu \left( {{E}} \right)}}} \right]^p}.$
\end{enumerate}
\end{lemma}
}
We now introduce the growth functions and their properties which can be found in \citep{Ky:2011, Yang:2012}.
\begin{definition}\label{Def.Growth_Fun}\rm{
A function $\varphi :X \times \left[ {0, + \infty } \right) \to \left[ {0, + \infty } \right)$ is called a \emph{growth function} if the following remain true:
\begin{enumerate}[\rm(i)]
  \item $\varphi$ is a \emph{Musielak-Orlicz function}, namely,
  \begin{enumerate}[\rm({i})$_1$]
  \setlength{\itemsep}{0ex} 
    \item the function $\varphi \left( {x, \cdot } \right): \left[ {0, + \infty } \right) \to \left[ {0, + \infty } \right)$ is an Orlicz function for all $x \in X$;
    \item b the funtion $\varphi \left( { \cdot ,t} \right)$ is an measurable function for all $t \in \left[0, \infty\right)$.
  \end{enumerate}
  \item $\varphi \in \mathbb{A}_{\infty}\left(X\right).$
  \item The function $\varphi$ is of positive uniformly upper type $p_1$ for some $p_{1} \in \left(0, 1\right]$ and of uniformly lower type $p_2$ for some $p_{2} \in \left(0, 1\right]$.
\end{enumerate}
}
\end{definition}
\begin{lemma}{}\label{Lemma_Properties_of_Growth_Functions_01}
Let $\varphi$ be a growth function.
 Set $\tilde \varphi \left( {x,t} \right): = \int_0^t {\varphi \left( {x,s} \right)\frac{{ds}}{s}} $ for all $\left(x, t\right)\in X\times \left[0, \infty\right)$. Then $\tilde \varphi$ is a growth function, which is equivalent to $\varphi$; moreover, $\tilde \varphi\left(x, \cdot\right)$ is continuous and strictly increasing.
\end{lemma}
\subsection{Musielak-Orlicz Space}
In this subsection we recall the Musielak-Orlicz Space and obtain a vector-valued inequality. In what follows, we always assume $\varphi$ is a \emph{growth function}.

The \emph{Musielak-Orlicz space} $L^{\varphi}\left(X\right)$ contains all measurable functions $f$ which satisfy $\int_X {\varphi \left( {x,\left| {f\left( x \right)} \right|} \right)d\mu \left( x \right)}  < \infty $ with \emph{Luxembourg norm}
$$
\setlength\abovedisplayskip{1ex}
\setlength\belowdisplayskip{1ex}
{\left\| f \right\|_{{L^\varphi }\left( X \right)}}: = inf\left\{ {\lambda  \in \left( {0,\infty } \right);\int_X {\varphi \left( {x,\frac{{\left| {f\left( x \right)} \right|}}{\lambda }} \right)d\mu \left( x \right)}  \le 1} \right\}.
$$

The following Lemma of Musielak-Orlicz Fefferman-Stein vector-valued inequality is obtained by Obtained by \citet{Yang:2012-new}. In what follows, the space ${L^\varphi }\left( {{\ell ^p},X} \right)$ is defined to be the set of all ${{{\left\{ {{f_j}} \right\}}_{j \in \mathbb{Z}}}}$ satisfying ${\left[ {\sum\nolimits_j {{{\left| {{f_j}} \right|}^p}} } \right]^{{1 \mathord{\left/
 {\vphantom {1 p}} \right.
 \kern-\nulldelimiterspace} p}}} \in {L^\varphi }\left( X \right)$ and we let
 $$
 \setlength\abovedisplayskip{1ex}
 \setlength\belowdisplayskip{1ex}
 {\left\| {{{\left\{ {{f_j}} \right\}}_{j \in \mathbb{Z}}}} \right\|_{{L^\varphi }\left( {{\ell ^p},X} \right)}}: = {\left\| {{{\left[ {\sum\nolimits_j {{{\left| {{f_j}} \right|}^p}} } \right]}^{{1 \mathord{\left/
 {\vphantom {1 p}} \right.
 \kern-\nulldelimiterspace} p}}}} \right\|_{{L^\varphi }\left( X \right)}}.
 $$
\begin{lemma}{}\label{Lemma_Properties_of_Growth_Functions_04}
Let $p \in \left(1, \infty\right]$, $\varphi$ be a Musielak-Orlicz function with uniformly lower type $p_1$ and upper type $p_2$, and $\varphi \in \mathbb{A}_{q}\left(X\right)$ for $q\in\left(1, \infty\right)$. If $q\left(\varphi\right) < p_1 \le p_2 < \infty$, then there exists a positive constant $C$ such that, for all ${{{\left\{ {{f_j}} \right\}}_{j \in \mathbb{Z}}}} \in {L^\varphi }\left( {{\ell ^p},X} \right)$,
$$
 \setlength\abovedisplayskip{1ex}
 \setlength\belowdisplayskip{1ex}
 \int_X {\varphi \left( {x,{{\left[ {\sum\nolimits_j {\mathcal{M}\left( {{f_j}} \right){{\left( x \right)}^p}} } \right]}^{{1 \mathord{\left/
 {\vphantom {1 p}} \right.
 \kern-\nulldelimiterspace} p}}}} \right)d\mu \left( x \right)}  \le C\int_X {\varphi \left( {x,{{\left[ {\sum\nolimits_j {{{\left| {{f_j}\left( x \right)} \right|}^p}} } \right]}^{{1 \mathord{\left/
 {\vphantom {1 p}} \right.
 \kern-\nulldelimiterspace} p}}}} \right)d\mu \left( x \right)}.
$$
\end{lemma}

\begin{corollary}{}\label{Corollary_Property_of_Growth_Functions_04}
Let $p$ and $\varphi$ be as in Lemma \ref{Lemma_Properties_of_Growth_Functions_04}, then for all $r \in \left( {0,{{{p_1}} \mathord{\left/
 {\vphantom {{{p_1}} {q\left( \varphi  \right)}}} \right.
 \kern-\nulldelimiterspace} {q\left( \varphi  \right)}}} \right)$, we have
$$
 \setlength\abovedisplayskip{1ex}
 \setlength\belowdisplayskip{1ex}
\int_X {\varphi \left( {x,{{\left[ {\sum\nolimits_j {\mathcal{M}\left( {{f_j}} \right){{\left( x \right)}^p}} } \right]}^{{1 \mathord{\left/
 {\vphantom {1 {rp}}} \right.
 \kern-\nulldelimiterspace} {rp}}}}} \right)d\mu \left( x \right)}  \le C\int_X {\varphi \left( {x,{{\left[ {\sum\nolimits_j {{{\left| {{f_j}\left( x \right)} \right|}^p}} } \right]}^{{1 \mathord{\left/
 {\vphantom {1 {rp}}} \right.
 \kern-\nulldelimiterspace} {rp}}}}} \right)d\mu \left( x \right)}.
$$
\end{corollary}
\begin{proof}{}
For any fixed $r \in \left( {0,{{{p_1}} \mathord{\left/
 {\vphantom {{{p_1}} {q\left( \varphi  \right)}}} \right.
 \kern-\nulldelimiterspace} {q\left( \varphi  \right)}}} \right)$, let $\tilde \varphi \left( {x,t} \right) = \varphi \left( {x,{t^{{1 \mathord{\left/
 {\vphantom {1 r}} \right.
 \kern-\nulldelimiterspace} r}}}} \right)$. We claim that $\tilde \varphi$ is of uniformly lower type $\frac{{p_{1}}}{r}$ and upper type $\frac{{p_{1}}}{r}$. By assumption, there exists a constant $C_{1}$, such that
 $$
 \setlength\abovedisplayskip{1ex}
 \setlength\belowdisplayskip{1ex}
 \tilde \varphi \left( {x,st} \right) = \varphi \left( {x,{s^{{1 \mathord{\left/
 {\vphantom {1 r}} \right.
 \kern-\nulldelimiterspace} r}}}{t^{{1 \mathord{\left/
 {\vphantom {1 r}} \right.
 \kern-\nulldelimiterspace} r}}}} \right) \le {C_1}{t^{{{{p_1}} \mathord{\left/
 {\vphantom {{{p_1}} r}} \right.
 \kern-\nulldelimiterspace} r}}}\varphi \left( {x,{s^{{1 \mathord{\left/
 {\vphantom {1 r}} \right.
 \kern-\nulldelimiterspace} r}}}} \right) = {C_1}{t^{{{{p_1}} \mathord{\left/
 {\vphantom {{{p_1}} r}} \right.
 \kern-\nulldelimiterspace} r}}}\tilde \varphi \left( {x,s} \right)
 $$
 for all $t \in \left[0, 1\right]$, $x \in X$ and $s \in \left[0, \infty\right)$. In the mean time, there exists another constant $C_2$, such that
 $$
 \setlength\abovedisplayskip{1ex}
 \setlength\belowdisplayskip{1ex}
 \tilde \varphi \left( {x,st} \right) = \varphi \left( {x,{s^{{1 \mathord{\left/
 {\vphantom {1 r}} \right.
 \kern-\nulldelimiterspace} r}}}{t^{{1 \mathord{\left/
 {\vphantom {1 r}} \right.
 \kern-\nulldelimiterspace} r}}}} \right) \le {C_2}{t^{{{{p_2}} \mathord{\left/
 {\vphantom {{{p_2}} r}} \right.
 \kern-\nulldelimiterspace} r}}}\varphi \left( {x,{s^{{1 \mathord{\left/
 {\vphantom {1 r}} \right.
 \kern-\nulldelimiterspace} r}}}} \right) = {C_2}{t^{{{{p_2}} \mathord{\left/
 {\vphantom {{{p_2}} r}} \right.
 \kern-\nulldelimiterspace} r}}}\tilde \varphi \left( {x,s} \right)
 $$
 for all $t \in \left[1, \infty\right)$, $x \in X$ and $s \in \left[0, \infty\right)$.

 Since $q\left(\tilde \varphi\right) = q\left(\varphi\right)$, for arbitrary $r \in \left( {0,{{{p_1}} \mathord{\left/
 {\vphantom {{{p_1}} {q\left( \varphi  \right)}}} \right.
 \kern-\nulldelimiterspace} {q\left( \varphi  \right)}}} \right)$, it is trivial that
 $$
 \setlength\abovedisplayskip{1ex}
 \setlength\belowdisplayskip{1ex}
 q\left( {\tilde \varphi } \right) = q\left( \varphi  \right) < \frac{{{p_1}}}{r} \le \frac{{{p_2}}}{r} < \infty.
 $$ By employing Lemma \ref{Lemma_Properties_of_Growth_Functions_04}, we obtain
 \begin{eqnarray*}
\int_X {\varphi \left( {x,{{\left[ {\sum\nolimits_j {\mathcal{M}\left( {{f_j}} \right){{\left( x \right)}^p}} } \right]}^{{1 \mathord{\left/
 {\vphantom {1 {rp}}} \right.
 \kern-\nulldelimiterspace} {rp}}}}} \right)d\mu \left( x \right)}  &=& \int_X {\tilde \varphi \left( {x,{{\left[ {\sum\nolimits_j {\mathcal{M}\left( {{f_j}} \right){{\left( x \right)}^p}} } \right]}^{{1 \mathord{\left/
 {\vphantom {1 p}} \right.
 \kern-\nulldelimiterspace} p}}}} \right)d\mu \left( x \right)} \\
 &\le& C\int_X {\tilde \varphi \left( {x,{{\left[ {\sum\nolimits_j {{{\left| {{f_j}\left( x \right)} \right|}^p}} } \right]}^{{1 \mathord{\left/
 {\vphantom {1 p}} \right.
 \kern-\nulldelimiterspace} p}}}} \right)d\mu \left( x \right)} \\
 &=& C\int_X {\varphi \left( {x,{{\left[ {\sum\nolimits_j {{{\left| {{f_j}\left( x \right)} \right|}^p}} } \right]}^{{1 \mathord{\left/
 {\vphantom {1 {rp}}} \right.
 \kern-\nulldelimiterspace} {rp}}}}} \right)d\mu \left( x \right)}.
 \end{eqnarray*}
\end{proof}

\subsection{\textbf{AT}$_{\mathcal{L},M}$-Family Associated with Operator $L$}
Recalling that $X$ is a space that satisfies the strong homogeneity property \eqref{Eq.Doubling_dim_of_X_01} with homogeneous dimension $n$. In the view of Lemma \ref{Lemma_Decompositon_of_Space}, there exists a collection $\left\{ {Q_\alpha ^k \subset X;k \in Z,} \right.$ $\left. {\alpha  \in {I_k}} \right\}$ of open subsets, where $I_{k}$ is the index set, such that for every $k \in \mathbb{Z}$,
$$
 \setlength\abovedisplayskip{1ex}
 \setlength\belowdisplayskip{1ex}
X = \bigcup\nolimits_{\alpha  \in {I_k}} {Q_\alpha ^k}
$$
  with properties of $Q_{\alpha}^{k}$ as in Lemma \ref{Lemma_Decompositon_of_Space}. In what follows, such open subsets $\left\{ {Q_\alpha ^k \subset X;k \in Z,} \right.$ $\left. {\alpha  \in {I_k}} \right\}$ is said to be a family of \emph{dyadic cubes} of $X$. And we now  turn to introduce the \textbf{AT}$_{\mathcal{L},M}$-family associated with an operator $L$ whose definition can also be found in \citep{Duong:2016}.
\begin{definition}\label{Def.AT_LM_Family}
Suppose that an operator $L$ satisfies (\textbf{H1}) and (\textbf{H2}) and $M \in \mathbb{N}$. A collection of functions ${\left\{ {{a_Q}} \right\}_{Q:Dyadic}}$ is said to be an \textbf{AT}$_{\mathcal{L},M}$-family associated with $L$, if for each dyadic cube $Q$, there exists a function $b_{Q}\in \mathcal{D}\left(L^{2M}\right)$ such that
\begin{enumerate}[\rm(i)]

  \item $a_Q = L^M \left(b_Q\right);$
  \item ${\mathop{\rm supp}\nolimits} \left( {{L^k}\left( {b{}_Q} \right)} \right) \subset 3Q$, $k = 0,1, \cdot  \cdot  \cdot ,2M$;
  \item ${\left( {\ell {{\left( Q \right)}^2}L} \right)^k}\left( {{b_Q}} \right) \le \ell {\left( Q \right)^{2M}}\mu {\left( Q \right)^{{{ - 1} \mathord{\left/
 {\vphantom {{ - 1} 2}} \right.
 \kern-\nulldelimiterspace} 2}}}$, $k = 0,1, \cdot  \cdot  \cdot ,2M$.
\end{enumerate}
\rm{Here, $\mathcal{D}\left(L\right)$ denotes the domain of operator $L$, and by $L^k$ the $k$-fold composition of $L$ with itself.}
\end{definition}

With this definition, we can decompose an $L^2$ function into \textbf{AT}$_{\mathcal{L},M}$-family. Given a function $f \in L^2\left(X\right)$, we say that $f$ has an \textbf{AT}$_{\mathcal{L},M}$-\emph{expansion}, if there exists sequence $s=\left\{s_{Q}\right\}_{Q:Dyadic}$, $0 \le s_Q < \infty$ and an \textbf{AT}$_{\mathcal{L}, M}$-family $\left\{a_Q\right\}_{Q:Dyadic}$ in $L^2\left(X\right)$ such that
{
 \setlength\abovedisplayskip{1ex}
 \setlength\belowdisplayskip{1ex}
\begin{equation}\label{Eq.Decomposition_L2_Function}
  f = \sum\limits_{Q:Dyadic} {{s_Q}{a_Q}}.
\end{equation}
We then define by $W_f\left(x\right)$ the function related to the sequence $s=\left\{s_{Q}\right\}_{Q:Dyadic}$, $0 \le s_Q < \infty$ as
\begin{equation}\label{Def.W_F}
  {W_f}\left( x \right): = {\left( {\sum\nolimits_{Q:Dyadic} {\mu {{\left( Q \right)}^{ - 1}}|{s_Q}{|^2}{\mathcal{X}_Q}\left( x \right)} } \right)^{{1 \mathord{\left/
 {\vphantom {1 2}} \right.
 \kern-\nulldelimiterspace} 2}}}.
\end{equation}
}
\begin{proposition}{}\label{Pro.ATLM_Family}
Given an operator $L$ that satisfies \rm(\textbf{H1})-\rm(\textbf{H2})\emph{ and $f \in L^2\left(X\right)$. Then for all $M\in \mathbb{N}$, $f$ has an} \rm{\textbf{AT}}\emph{$_{\mathcal{L}, M}$-expansion. Moreover, let $Q^{k}_{\alpha}$ and $\delta
$ be as in Lemma \ref{Lemma_Decompositon_of_Space}, we have}
$$
 \setlength\abovedisplayskip{1ex}
 \setlength\belowdisplayskip{1ex}
{s_{Q_\alpha ^k}} = {\left( {\int_{{\delta ^{k + 1}}}^{{\delta ^k}} {\int_{Q_\alpha ^k} {|{t^2}L{e^{ - {t^2}L}}f(y){|^2}d\mu \left( y \right)} \frac{{dt}}{t}} } \right)^{{1 \mathord{\left/
 {\vphantom {1 2}} \right.
 \kern-\nulldelimiterspace} 2}}}.
$$
\end{proposition}

The proof of Proposition \ref{Pro.ATLM_Family} can be found in \citep[Proof of Theorem 3.2]{Duong:2016}. We omit the details.
\section{Musielark-Orlicz Hardy Space ${H_{\varphi, L}}$ and its Equivalent Characterization}
In this section, we begin to study the Musielar-Orlicz-Hardy spce, and in what follows, we always assume the operator $L$ satisfies \rm{(\textbf{H1})} and \rm{(\textbf{H2})}, and $\varphi$ is a growth function which is defined in Definition \ref{Def.Growth_Fun}. With some basic notations set forth in Sect. 2, we first establish the following characterization for the Hardy space $H_{\varphi, L}$.
\begin{theorem}{}\label{Theorem_3_2}
  Suppose $L$ is an operator that satisfies \rm{(\textbf{H1})} \emph{and} \rm{(\textbf{H2})}. \emph{Let $\varphi$ be a growth function with uniformly lower type $p_1$ and $f \in {H_{\varphi, L}}\left( X \right) \cap {L^2}\left( X \right)$, then for all natural number $M > {{nq\left( \varphi  \right)} \mathord{\left/
 {\vphantom {{nq\left( \varphi  \right)} {\left( {2{p_1}} \right)}}} \right.
 \kern-\nulldelimiterspace} {\left( {2{p_1}} \right)}}$, $f$ has an \textbf{AT}$_{\mathcal{L}, M}$-expansion such that}
    $$
    \setlength\abovedisplayskip{1ex}  
    \setlength\belowdisplayskip{1ex}
    {\left\| f \right\|_{{H_{\varphi, L }}\left(X\right)}} \cong {\left\| {{W_f}} \right\|_{{L^\varphi }\left(X\right)}}.$$
\end{theorem}

Before we prove Theorem \ref{Theorem_3_2}, we need to introduce some notions and establish some results as follows.

For any $v \in \left(0, \infty\right)$ and $x \in X$, let ${\Gamma _\nu }\left( x \right): = \left\{ {\left( {y,t} \right) \in X \times \left( {0,\infty } \right);d\left( {x,y} \right) < \nu t} \right\}$ be the \emph{cone of aperture} $\nu$ \emph{with vertex} $x \in X$. For any closed subset $F$ of $X$, denote by $\mathcal{R}_\nu \left(F\right)$ the \emph{union of all cones with vertices in} $F$, i.e., ${\mathcal{R}_\nu }\left( F \right) = \bigcup\nolimits_{x \in F} {{\Gamma _\nu }\left( x \right)} $. In what follows, we denote $\Gamma_1\left(x\right)$ and $\mathcal{R}_1\left(F\right)$ simply by $\Gamma\left(x\right)$ and $\mathcal{R}\left(F\right)$, respectively. For any open subset $O$ of $X$, we establish the following geometric property of $\mathcal{R}\left({O^\complement}\right)$ which generalizes a similar result obtained by \citet[Lemma 1]{Nestor:1977} in the case of Euclidean space.
\begin{lemma}{}\label{Lemma_3_4}
Suppose that $\left(X, d, \mu\right)$ is a space of homogeneous type with constant $C_D > 1$ such that \eqref{Eq.Property_of_Homogeneous_Type} holds. Let $O$ be an open subset of $X$ , $F = O^\complement$ and $\mathcal{X}_O$ its characteristic function. If for $\nu > 1$, we define ${O^ * }$ as
$$
\setlength\abovedisplayskip{1ex}  
\setlength\belowdisplayskip{1ex}
{O^ * }: = \left\{ {x \in X;\mathcal{M}\left( {{\mathcal{X}_O}} \right)\left( x \right) > {{\left( {4\nu } \right)}^{ - 2{{\log }_2}{C_D}}}} \right\}
$$
and let $F^* = (O^*)^\complement$, then we have
\begin{enumerate}[\rm(i)]\setlength{\itemsep}{-1ex}
  \item $\mathcal{R}_\nu\left(F^*\right)$ is contained in $\mathcal{R}\left(F\right)$.
  \item If $\left(z, t\right) \in \mathcal{R}_\nu\left(F^*\right)$, then there exists some constant $C_\nu$ such that
  $$
  \setlength\abovedisplayskip{1ex}  
  \setlength\belowdisplayskip{1ex}
  V\left(z, t\right) < C_\nu \mu\left(B\left(z, t\right) \cap F\right).
  $$
\end{enumerate}
\end{lemma}
\begin{proof}{}
 The lemma is trivial if $\mathcal{R}_\nu\left(F^*\right)=\varnothing$. We then with no loss of generality assume that $\mathcal{R}_\nu\left(F^*\right) \ne \varnothing$, which implies that $O \ne X$. We then first prove (i). If $\left(z, t\right)\in \mathcal{R}_\nu\left(F^*\right)$, then either $z \in F$ or $z \in O$. In the first case it is apparent that $\left(z, t\right)\in \mathcal{R}\left(F\right)$, since $d\left(z, z\right) = 0 < t$.

 If we are in the second case, i.e., $z\in O$, let $\delta$ be the distance from $z$ to the closed and non-empty set $F$. This number $\delta$ is positive and finite, and $B\left(z, \delta\right)$ is contained in $O$. The assumption that $\left(z, t\right)\in \mathcal{R}_\nu\left(F^*\right)$ implies that there is $y \in F^*$ with $d\left(z, y\right) <\nu t$. Thus, writing $r = \delta + d\left(z, y\right)$, we get $B\left(z, \delta\right) \subset B\left(y, r\right)$ and also
 $$
  \setlength\abovedisplayskip{1ex}  
  \setlength\belowdisplayskip{1ex}
  B\left( {z,\delta } \right) \subset B\left( {z,\delta } \right) \cap O \subset B\left( {y,r} \right) \cap O,
 $$
 which together with the definition of $O^*$, implies that
 $$
  \setlength\abovedisplayskip{1ex}  
  \setlength\belowdisplayskip{1ex}
 V\left( {z,\delta } \right) \le \mu \left( {B\left( {y,r} \right) \cap O} \right) \le {\left( {4\nu } \right)^{ - 2{{\log }_2}{C_D}}}V\left( {y,r} \right)
 $$
 since $y \in F^*$.

 By using \eqref{Eq.Doubling_dim_of_X_02} twice, we also have
 {
 \setlength\abovedisplayskip{1ex}  
 \setlength\belowdisplayskip{1ex}
 \begin{align*}
V\left( {y,r} \right) &\le {C_D}{\left( {r{\delta ^{ - 1}}} \right)^{{{\log }_2}{C_D}}}V\left( {y,\delta } \right)\\
 &\le C_D^2{\left( {r{\delta ^{ - 1}}} \right)^{{{\log }_2}{C_D}}}{\left( {1 + {\delta ^{ - 1}}d\left( {y,z} \right)} \right)^{{{\log }_2}{C_D}}}V\left( {z,\delta } \right)\\
 &= {\left( {2r{\delta ^{ - 1}}} \right)^{2{{\log }_2}{C_D}}}V\left( {z,\delta } \right).
 \end{align*}
 From these inequalities, we get that
 }
 $$
 \setlength\abovedisplayskip{1ex}  
 \setlength\belowdisplayskip{1ex}
 \delta \le \frac{r}{{2\nu }}.
 $$
 Recalling that $r = \delta + d\left(z, y\right)$ and $d\left(z, y\right)<\nu t$, we obtain
 $$
 \setlength\abovedisplayskip{1ex}  
 \setlength\belowdisplayskip{1ex}
 \delta  \le \frac{{\delta  + d\left( {z,y} \right)}}{{2\nu }} < \frac{{\delta  + \nu t}}{{2\nu }}
 $$
 and since $\nu > 1$, it follows that $\delta < t$. Then by the very definition of $\delta$, there exists an $x \in F$, satisfying $d\left(x, z\right) < t$, which means that $\left(z, t\right)\in \mathcal{R}\left(F\right)$. This proves (i).

 We then turn our attention to (ii). If $\left(z, t\right) \in \mathcal{R}_\nu \left(F^*\right)$, there is $y \in F^*$ such that $d\left(z, y\right) < \nu t$. Then $B\left(z, t\right) \subset B\left(y, \left(1 + \nu\right)t\right)$ and since $y \in F^*$, we get
 $$
 \setlength\abovedisplayskip{1ex}  
 \setlength\belowdisplayskip{1ex}
 \mu \left( {B\left( {z,t} \right) \cap O} \right) \le \mu \left( {B\left( {y,\left( {1 + \nu } \right)t} \right) \cap O} \right) \le {\left( {4\nu } \right)^{ - 2{{\log }_2}{C_D}}}V\left( {y,\left( {1 + \nu } \right)t} \right),
 $$
 and therefore
 {
 \setlength\abovedisplayskip{1ex}  
 \setlength\belowdisplayskip{1ex}
 \begin{align*}
\mu \left( {B\left( {z,t} \right) \cap O} \right) &\le {\left( {4\nu } \right)^{ - 2{{\log }_2}{C_D}}}{C_D}{\left( {1 + \nu } \right)^{{{\log }_2}{C_D}}}V\left( {y,t} \right)\\
 &\le C_D^2{\left( {4\nu } \right)^{ - 2{{\log }_2}{C_D}}}{\left( {1 + \nu } \right)^{{{\log }_2}{C_D}}}{\left( {1 + {t^{ - 1}}d\left( {y,z} \right)} \right)^{{{\log }_2}{C_D}}}V\left( {z,t} \right)\\
 &< {\left( {\frac{{1 + \nu }}{{2\nu }}} \right)^{2{{\log }_2}{C_D}}}V\left( {z,t} \right).
 \end{align*}
 Now from $V\left( {z,t} \right) = \mu \left( {B\left( {z,t} \right) \cap O} \right) + \mu \left( {B\left( {z,t} \right) \cap F} \right)$, we obtain
 }
 $$
 \setlength\abovedisplayskip{1ex}  
 \setlength\belowdisplayskip{1ex}
 \left[ {1 - {{\left( {\frac{{1 + \nu }}{{2\nu }}} \right)}^{2{{\log }_2}{C_D}}}} \right]V\left( {z,t} \right) < \mu \left( {B\left( {z,t} \right) \cap F} \right),
 $$
 which implies (ii).

\end{proof}

Next we introduce the following variant of Lusin-area function associated with $L$. For all $\nu \in \left(0, \infty\right)$, $f \in L^2\left(X\right)$ and $x \in X$, let
$$
 \setlength\abovedisplayskip{1ex}  
 \setlength\belowdisplayskip{1ex}
 {S_{L,\nu }}\left( f \right)\left( x \right): = {\left( {\int_0^\infty  {\int_{d\left( {x,y} \right) < \nu t} {|{t^2}L{e^{ - {t^2}L}}\left( f \right)(y){|^2}\frac{{d\mu \left( y \right)}}{{V\left( {x,t} \right)}}} \frac{{dt}}{t}} } \right)^{{1 \mathord{\left/
 {\vphantom {1 2}} \right.
 \kern-\nulldelimiterspace} 2}}}.
$$
We also have the following two Lemmas for the variant of Lusin-area function that associated with $L$ which generalize the results of \citet[Lemma 2]{Nestor:1977} and \citet[Lemma 3.3.5]{Yang:2017}.
\begin{lemma}{}\label{Lemma_3_5}
 Assume that L satisfies \rm{(\textbf{H1})} and \rm{(\textbf{H2})}. \emph{Let $\varphi \in \mathbb{A}_p\left(X\right)$, $1\le p < \infty$, and O be an open subset of $X$. If $O^*$ is the set associated to $O$ as in Lemma \ref{Lemma_3_4} with some $\nu > 1$, then there exists a finite constant $C$, which is independent of $O$, such that for all $\lambda \in \left(0, \infty\right)$ and $f \in L^2\left(X\right)$ },
 $$
 \setlength\abovedisplayskip{1ex}  
 \setlength\belowdisplayskip{1ex}
 \int_{{F^*}} {|{S_{L,\nu }}\left( f \right)(x){|^2}\varphi \left( {x,\lambda } \right)d\mu \left( x \right)}  \le C\int_F {|{S_L}\left( f \right)(x){|^2}\varphi \left( {x,\lambda } \right)d\mu \left( x \right)},
 $$
 where $F^* = \left(O^*\right)^\complement$ and $F = O^\complement$.
\end{lemma}
\begin{proof}{}
For any $x \in F^*$, $\left(y, t\right)\in \Gamma_\nu \left(x\right)$, we observe that $d\left(x, y\right) < \nu t$, and hence by \eqref{Eq.Doubling_dim_of_X_02},
{
\begin{align*}
V{\left( {x,t} \right)^{ - 1}} &\le {C_D}{\left( {1 + {t^{ - 1}}d\left( {x,y} \right)} \right)^{{{\log }_2}{C_D}}}V{\left( {y,t} \right)^{ - 1}}\\
 &< {C_D}{\left( {1 + \nu } \right)^{{{\log }_2}{C_D}}}V{\left( {y,t} \right)^{ - 1}}.
\end{align*}
It follows that
}
{
 \setlength\abovedisplayskip{1ex}  
 \setlength\belowdisplayskip{1ex}
 \begin{align}\label{Eq.Proof_Lemma_3_5_01}
&\int_{{F^*}} {|{S_{L,\nu }}\left( f \right)(x){|^2}\varphi \left( {x,\lambda } \right)d\mu \left( x \right)} \nonumber\\&
\begin{array}{*{20}{c}}
{}&{}&{}
\end{array} \le {C_D}{\left( {1 + \nu } \right)^{{{\log }_2}{C_D}}}\int_{{F^*}} {\left( {\int_{{\Gamma _\nu }\left( x \right)} {|{t^2}L{e^{ - {t^2}L}}\left( f \right)(y){|^2}\frac{{d\mu \left( y \right)dt}}{{V\left( {y,t} \right)t}}} } \right)\varphi \left( {x,\lambda } \right)d\mu \left( x \right)} \nonumber\\&
\begin{array}{*{20}{c}}
{}&{}&{}
\end{array} = {C_{\nu ,D}}\int_{{\mathcal{R}_\nu }({F^*})} {|{t^2}L{e^{ - {t^2}L}}\left( f \right)(y){|^2}V{{\left( {y,t} \right)}^{ - 1}}\varphi \left( {B\left( {y,\nu t} \right) \cap {F^*},\lambda } \right)\frac{{d\mu \left( y \right)dt}}{t}}.
 \end{align}
 We then employ Lemma \ref{Lemma_Properties_of_Growth_Functions_03} (iv) to the set $E = B\left(y,t\right)$ and $B = B\left(y, \nu t\right)$, to get
}
{
 \setlength\abovedisplayskip{1ex}  
 \setlength\belowdisplayskip{1ex}
 \begin{equation}\label{Eq.Proof_Lemma_3_5_02}
   \varphi \left( {B\left( {y,\nu t} \right),\lambda } \right) \le C{\left( {2\nu } \right)^{{{\log }_2}{C_D}}}\varphi \left( {B\left( {y,t} \right),\lambda } \right).
 \end{equation}
 Applying Lemma \ref{Lemma_Properties_of_Growth_Functions_03} (iv) once again to $E = B\left( {y,t} \right) \cap F$ and $B = B\left(y, t\right)$, we get
}
{
 \setlength\abovedisplayskip{1ex}  
 \setlength\belowdisplayskip{1ex}
 \begin{equation}\label{Eq.Proof_Lemma_3_5_03}
   \varphi \left( {B\left( {y,t} \right),\lambda } \right) \le C{\left( {\frac{{V\left( {y,t} \right)}}{{\mu \left( {B\left( {y,t} \right) \cap F} \right)}}} \right)^p}\varphi \left( {B\left( {y,t} \right) \cap F,\lambda } \right).
 \end{equation}
 Therefore, from \eqref{Eq.Proof_Lemma_3_5_02} and \eqref{Eq.Proof_Lemma_3_5_03}, plus part (ii) of Lemma \ref{Lemma_3_4}, we have
}
$$
 \setlength\abovedisplayskip{1ex}  
 \setlength\belowdisplayskip{1ex}
 \varphi \left( {B\left( {y,\nu t} \right),\lambda } \right) \le C\varphi \left( {B\left( {y,t} \right) \cap F,\lambda } \right).
$$
From this estimate it follows that the last integral in \eqref{Eq.Proof_Lemma_3_5_01} is bounded by
$$
 \setlength\abovedisplayskip{1ex}  
 \setlength\belowdisplayskip{1ex}
 C\int_{{\mathcal{R}_\nu }({F^*})} {|{t^2}L{e^{ - {t^2}L}}\left( f \right)(y){|^2}V{{\left( {y,t} \right)}^{ - 1}}\varphi \left( {B\left( {y,t} \right) \cap F,\lambda } \right)\frac{{d\mu \left( y \right)dt}}{t}}.
$$
Finally, in the view of Lemma \ref{Lemma_3_4} (i), we observe that $\mathcal{R}_\nu \left(F^*\right) \subset \mathcal{R}\left(F\right)$, it follows immediately that the last integral above is bounded by
{
 \setlength\abovedisplayskip{1ex}  
 \setlength\belowdisplayskip{1ex}
\begin{align*}
&C\int_{\mathcal{R}(F)} {|{t^2}L{e^{ - {t^2}L}}\left( f \right)(y){|^2}V{{\left( {y,t} \right)}^{ - 1}}\varphi \left( {B\left( {y,t} \right) \cap F,\lambda } \right)\frac{{d\mu \left( y \right)dt}}{t}} \\&
\begin{array}{*{20}{c}}
{}&{}&{}
\end{array} = C\int_F {\left( {\int_{{\Gamma _\nu }\left( x \right)} {|{t^2}L{e^{ - {t^2}L}}\left( f \right)(y){|^2}\frac{{d\mu \left( y \right)dt}}{{V\left( {y,t} \right)t}}} } \right)\varphi \left( {x,\lambda } \right)d\mu \left( x \right)} \\&
\begin{array}{*{20}{c}}
{}&{}&{}
\end{array} \le C\int_F {|{S_L}\left( f \right)(x){|^2}\varphi \left( {x,\lambda } \right)d\mu \left( x \right)},
\end{align*}
where the last inequality follows from the fact that
}
$$
 \setlength\abovedisplayskip{1ex}  
 \setlength\belowdisplayskip{1ex}
V{\left( {y,t} \right)^{ - 1}} \le {C_D}{\left( {1 + {t^{ - 1}}d\left( {x,y} \right)} \right)^{{{\log }_2}{C_D}}}V{\left( {x,t} \right)^{ - 1}} < C_D^2V{\left( {x,t} \right)^{ - 1}}
$$
for $\left( {y,t} \right) \in {\Gamma _\nu }\left( x \right)$. This proves the lemma.

\end{proof}
\begin{lemma}{}\label{Lemma_3_6}
Assume that L satisfies \rm{(\textbf{H1})} and \rm{(\textbf{H2})}. \emph{Let $q\in\left(1, \infty\right)$, $\varphi$ be as in Definition \ref{Def.Growth_Fun} and $\varphi \in \mathbb{A}_q\left(X\right)$. Then for all $\nu \in \left(0, \infty\right)$ there exists a positive constant $C_\nu$ such that, for all measurable functions $f$, }
$$
 \setlength\abovedisplayskip{1ex}  
 \setlength\belowdisplayskip{1ex}
 \int_X {\varphi \left( {x,{S_{L,\nu }}\left( f \right)(x)} \right)d\mu \left( x \right)}  \le C_\nu \int_X {\varphi \left( {x,{S_L}\left( f \right)(x)} \right)d\mu \left( x \right)}.
$$
\end{lemma}
\begin{proof}{}
  If $\nu \in \left(0, 1\right]$, the conclusion is trivial. We further suppose that $\nu \in \left(1, \infty\right)$. For all $\lambda \in \left(0, \infty\right)$, let
  $$
  \setlength\abovedisplayskip{1ex}  
  \setlength\belowdisplayskip{1ex}
  {O_\lambda }: = \left\{ {x \in X;{S_L}\left( f \right)(x) > \lambda } \right\}
  $$
  and
  $$
  \setlength\abovedisplayskip{1ex}  
  \setlength\belowdisplayskip{1ex}
  O_\lambda ^*: = \left\{ {x \in X;\mathcal{M}\left( {{\mathcal{X}_{{O_\lambda }}}} \right)(x) > {{\left( {4\nu } \right)}^{ - {{\log }_2}{C_D}}}} \right\}.
  $$
  where $\mathcal{M}$ is the Hardy-Littlewood maximal function. Since $\varphi \in \mathbb{A}_q\left(X\right)$, it follows from Lemma \ref{Lemma_Properties_of_Growth_Functions_03} (iii),
  {
  \setlength\abovedisplayskip{1ex}  
  \setlength\belowdisplayskip{1ex}
  \begin{align}\label{Eq.Proof_Lemma_3_6_01}
\varphi \left( {O_\lambda ^*,\lambda } \right) &= \varphi \left( {\left\{ {x \in X;\mathcal{M}\left( {{\mathcal{X}_{{O_\lambda }}}} \right)(x) > {{\left( {4\nu } \right)}^{ - {{\log }_2}{C_D}}}} \right\},\lambda } \right)\nonumber\\
 &\le \int_X {{{\left( {4\nu } \right)}^{q{{\log }_2}{C_D}}}{{\left( {\mathcal{M}\left( {{\mathcal{X}_{{O_\lambda }}}} \right)(x)} \right)}^q}\varphi \left( {x,\lambda } \right)d\mu \left( x \right)} \nonumber\\
 &\le C\varphi \left( {{O_\lambda },\lambda } \right).
  \end{align}
  Let $F_\lambda:=O_\lambda^\complement$, $F_\lambda^*:=(O_\lambda^*)^\complement$ and apply Lemma \ref{Lemma_3_5} to get
  }
  {
  \setlength\abovedisplayskip{1ex}  
  \setlength\belowdisplayskip{1ex}
  \begin{equation}\label{Eq.Proof_Lemma_3_6_02}
    \int_{{F_\lambda^*}} {|{S_{L,\nu }}\left( f \right)(x){|^2}\varphi \left( {x,\lambda } \right)d\mu \left( x \right)}  \le C\int_{F_\lambda} {|{S_L}\left( f \right)(x){|^2}\varphi \left( {x,\lambda } \right)d\mu \left( x \right)}.
  \end{equation}
  Thus, from \eqref{Eq.Proof_Lemma_3_6_01} and \eqref{Eq.Proof_Lemma_3_6_02}, it follows that
  }
  {
  \setlength\abovedisplayskip{1ex}  
  \setlength\belowdisplayskip{1ex}
  \begin{align*}
&\varphi \left( {\left\{ {x \in X;{S_{L,\nu }}\left( f \right)(x) > \lambda } \right\},\lambda } \right)\\&
\begin{array}{*{20}{c}}
{}&{}&{}&{}
\end{array} \le \varphi \left( {O_\lambda ^*,\lambda } \right) + \varphi \left( {\left\{ {x \in F_\lambda ^*;{S_{L,\nu }}\left( f \right)(x) > \lambda } \right\},\lambda } \right)\\&
\begin{array}{*{20}{c}}
{}&{}&{}&{}
\end{array} \le C\varphi \left( {{O_\lambda },\lambda } \right) + {\lambda ^{ - 2}}\int_{F_\lambda ^*} {|{S_{L,\nu }}\left( f \right)(x){|^2}\varphi \left( {x,\lambda } \right)d\mu \left( x \right)} \\&
\begin{array}{*{20}{c}}
{}&{}&{}&{}
\end{array} \le C\left[ {\varphi \left( {{O_\lambda },\lambda } \right) + {\lambda ^{ - 2}}\int_{{F_\lambda }} {|{S_L}\left( f \right)(x){|^2}\varphi \left( {x,\lambda } \right)d\mu \left( x \right)} } \right]\\&
\begin{array}{*{20}{c}}
{}&{}&{}&{}
\end{array} \le C\left[ {\varphi \left( {{O_\lambda },\lambda } \right) + {\lambda ^{ - 2}}\int_0^\lambda  {t\varphi \left( {\left\{ {x \in X;{S_L}\left( f \right)(x) > t} \right\},\lambda } \right)dt} } \right],
  \end{align*}
  which, together with the assumption $\nu \in \left(1, \infty\right)$, Lemma \ref{Lemma_Properties_of_Growth_Functions_01} and the uniformly upper type $1$ of $\varphi$, we further get that
  }
  {
  \setlength\abovedisplayskip{1ex}  
  \setlength\belowdisplayskip{1ex}
  \begin{align*}
&\int_X {\varphi \left( {x,{S_{L,\nu }}\left( f \right)(x)} \right)d\mu \left( x \right)} \\&
\begin{array}{*{20}{c}}
{}&{}&{}&{}
\end{array} \le C \int_0^\infty  {{\lambda ^{ - 1}}\varphi \left( {\left\{ {x \in X;{S_{L,\nu }}\left( f \right)(x) > \lambda } \right\},\lambda } \right)d\lambda } \\&
\begin{array}{*{20}{c}}
{}&{}&{}&{}
\end{array} \le C\left[ {\int_0^\infty  {{\lambda ^{ - 1}}\varphi \left( {{O_\lambda },\lambda } \right)d\lambda } } \right.\\&
\begin{array}{*{20}{c}}
{}&{}&{}&{}&{}&{}&{}&{}
\end{array}\left. { + \int_0^\infty  {{\lambda ^{ - 3}}\int_0^\lambda  {t\varphi \left( {\left\{ {x \in X;{S_L}\left( f \right)(x) > t} \right\},\lambda } \right)dt} d\lambda } } \right]\\&
\begin{array}{*{20}{c}}
{}&{}&{}&{}
\end{array} \le C\left[ {\int_0^\infty  {{\lambda ^{ - 1}}\varphi \left( {\left\{ {x \in X;{S_L}\left( f \right)(x) > \lambda } \right\},\lambda } \right)d\lambda } } \right.\\&
\left. {\begin{array}{*{20}{c}}
{}&{}&{}&{}&{}&{}&{}&{}
\end{array} + \int_0^\infty  {{\lambda ^{ - 3}}\int_0^\lambda  {t\varphi \left( {\left\{ {x \in X;{S_L}\left( f \right)(x) > t} \right\},\lambda } \right)dt} d\lambda } } \right]\\&
\begin{array}{*{20}{c}}
{}&{}&{}&{}
\end{array} \le C\left[ {\int_X {\varphi \left( {x,{S_L}\left( f \right)(x)} \right)d\mu \left( x \right)} } \right.\\&
\begin{array}{*{20}{c}}
{}&{}&{}&{}&{}&{}&{}&{}
\end{array}\left. { + \int_0^\infty  {\varphi \left( {\left\{ {x \in X;{S_L}\left( f \right)(x) > t} \right\},t} \right)\int_t^\infty  {{\lambda ^{ - 2}}d\lambda } dt} } \right]\\&
\begin{array}{*{20}{c}}
{}&{}&{}&{}
\end{array} \le C\int_X {\varphi \left( {x,{S_L}\left( f \right)(x)} \right)d\mu \left( x \right)}.
  \end{align*}
  This finishes the proof of Lemma \ref{Lemma_3_6}.
  }

\end{proof}

Moreover, we also need the following Lemma, whose standard proof can be found in \citep{Frazier:1990}, we omit the details. And in what follows, we recall that the Hardy-Littlewood maximal operator $\mathcal{M}$ on $\left( X, \mu, d\right)$ is defined by
$$
\setlength\abovedisplayskip{1ex}
\setlength\belowdisplayskip{1ex}
\mathcal{M}\left( f \right)\left( x\right): = \mathop {\sup }\limits_{B \ni x} \frac{1}{{\mu\left( B \right)}}\int_B {\left| {f\left( y \right)} \right|d\mu \left( y \right)}.
$$
where the supremum is taken over all balls $B \ni x$.
\begin{lemma}{}\label{Lemma_3_2}
Suppose $0 < q \le 1$ and $N > {n \mathord{\left/
 {\vphantom {n q}} \right.
 \kern-\nulldelimiterspace} q}$, where $n$ is the doubling dimension of the space in \eqref{Eq.Doubling_dim_of_X_01}. Fix an integer $k$, and let ${\left\{ {{s_{Q_\alpha ^k}}} \right\}_{\alpha  \in {I_k}}}$ be as in Proposition \ref{Pro.ATLM_Family}, then for any subsequence ${I_k}^\prime  \subset {I_k}$ and each $x \in X$,
 $$
 \setlength\abovedisplayskip{1ex}
 \setlength\belowdisplayskip{1ex}
 \sum\limits_{\alpha  \in {I_k}^\prime } {\frac{{\left| {{s_{Q_\alpha ^k}}} \right|}}{{{{\left[ {1 + \ell{{( {Q_\alpha ^k})}^{ - 1}}d\left( {x, y_\alpha ^k} \right)} \right]}^N}}}}  \le C{\left[ {\mathcal{M}\left( {\sum\limits_{\alpha  \in {I_k}^\prime } {{{\left| {{s_{Q_\alpha ^k}}} \right|}^q}\mathcal{X}\left(  \cdot  \right)} } \right)} \right]^{{1 \mathord{\left/
 {\vphantom {1 q}} \right.
 \kern-\nulldelimiterspace} q}}}\left( x \right).
 $$
 where $y_{\alpha}^k$ denotes the center of $Q_{\alpha}^k$.
\end{lemma}
\begin{proofoftheorem}{\ref{Theorem_3_2}}
  For any fixed $f \in {H_{\varphi, L}}\left( X \right) \cap {L^2}\left( X \right)$, we let ${\lambda _0} = {\left\| f \right\|_{{H_{\varphi, L}}\left(X\right)}}$ and ${\lambda _1} = {\left\| {{W_f}} \right\|_{{L^\varphi }\left(X\right)}}$. It suffices to show that for all $\lambda \in \left(0, \infty \right)$,
  {
  \setlength\abovedisplayskip{1ex}  
  \setlength\belowdisplayskip{1ex}
  \begin{equation}\label{Proof_of_Theorem_3_2_01}
    \int_X {\varphi \left( {x,\frac{{{S_L}\left( f \right)\left( x \right)}}{\lambda }} \right)d\mu \left( x \right)}  \cong \int_X {\varphi \left( {x,\frac{{\left| {{W_f}\left( x \right)} \right|}}{\lambda }} \right)d\mu \left( x \right)}.
  \end{equation}
  In fact, if \eqref{Proof_of_Theorem_3_2_01} holds for all $\lambda \in \left(0, \infty \right)$, then there exists a constant $C_0$ such that
  }
  $$
  \setlength\abovedisplayskip{1ex}  
  \setlength\belowdisplayskip{1ex}
  \int_X {\varphi \left( {x,\frac{{{S_L}\left( f \right)\left( x \right)}}{\lambda _1 }} \right)d\mu \left( x \right)}  \le {C_0}\int_X {\varphi \left( {x,\frac{{\left| {{W_f}\left( x \right)} \right|}}{{{\lambda _1}}}} \right)d\mu \left( x \right)}  \le {C_0} ,
  $$
  which, together with \eqref{Eq.Property_of_Growth_Function}, implies that
  $$
  \setlength\abovedisplayskip{1ex}  
  \setlength\belowdisplayskip{1ex}
  \int_X {\varphi \left( {x,\frac{{{G_L}\left( f \right)\left( x \right)}}{{{C_1}{\lambda _1}}}} \right)d\mu \left( x \right)}  \le 1
  $$
  for some constants $C_1$, and hence we have $\lambda _0 \le C_1 \lambda _1$. In a similar fashion, one can prove $\lambda _1 \le C_2 \lambda _0$ for some constants $C_2$, and get the desired property.

  Let $f$ be a function in $ {H_{\varphi, L }}\left( X \right) \cap {L^2}\left( X \right)$. In the view of Lemma \ref{Lemma_Decompositon_of_Space}, for any fixed $\left( x, k\right) \in X \times \mathbb{Z}$ there exists a unique $\alpha \in I_{k}$, such that $x \in Q_\alpha ^k$. Let $Q_x ^k$ denote the such $Q_\alpha ^k$ and we write
  {
  \setlength\abovedisplayskip{1ex}  
  \setlength\belowdisplayskip{1ex}
\begin{align}\label{Proof_of_Theorem_3_2_02}
{W_f}\left( x \right)
 &= {\left\{ {\sum\limits_{k \in \mathbb{Z}} {\sum\limits_{\alpha  \in {I_k}} {{{\left[ {\mu{{\left( {Q_\alpha ^k} \right)}^{ - {1 \mathord{\left/
 {\vphantom {1 2}} \right.
 \kern-\nulldelimiterspace} 2}}}\left| {{s_{Q_\alpha ^k}}} \right|{\mathcal{X} _{Q_\alpha ^k}}\left( x \right)} \right]}^2}} } } \right\}^{{1 \mathord{\left/
 {\vphantom {1 2}} \right.
 \kern-\nulldelimiterspace} 2}}} \nonumber\\
 &= {\left\{ {\sum\limits_{k \in \mathbb{Z}} {\mu{{\left( {Q_x^k} \right)}^{ - 1}}{{\left| {{s_{Q_x^k}}} \right|}^2}} } \right\}^{{1 \mathord{\left/
 {\vphantom {1 2}} \right.
 \kern-\nulldelimiterspace} 2}}}\nonumber\\
 &= {\left\{ {\sum\limits_{k \in \mathbb{Z}} {\int_{{\delta ^{k + 1}}}^{{\delta ^k}} {\mu{{\left( {Q_x^k} \right)}^{ - 1}}\int_{Q_x^k} {{{| {{t^2}L{e^{ - {t^2}L}}f\left( y \right)} |}^2}d\mu \left( y \right)\frac{{dt}}{t}} } } } \right\}^{{1 \mathord{\left/
 {\vphantom {1 2}} \right.
 \kern-\nulldelimiterspace} 2}}},
\end{align}
where $\delta \in \left( 0, 1\right)$ is a constant as in Lemma \ref{Lemma_Decompositon_of_Space} and the last quantity follows from Proposition \ref{Pro.ATLM_Family}. Moreover, by $\left( {iv} \right)$ and $\left( {v} \right)$ of Lemma \ref{Lemma_Decompositon_of_Space}, we know that for any fixed $\left( x, k\right) \in X \times \mathbb{Z}$ there exists $z_x^k \in Q_x^k$ and constants $C_1 \in \left( 0, 1\right)$, $C_2 = \delta ^{-1}$ such that ${\rm{diam}}\left( {Q_\alpha ^k} \right) \le {C_2}{\delta ^k}:= \ell\left(Q_\alpha^k\right)$ and
}
$$
\setlength\abovedisplayskip{1ex}  
\setlength\belowdisplayskip{1ex}
B\left( {z_x^k,{C_1}{\delta ^k}} \right) \subset Q_x^k \subset B\left( {x,{C_2}{\delta ^k}} \right) \subset B\left( {x,{C_2}{\delta ^{ - 1}}t} \right)
$$
for all $t \in \left( {{\delta ^{k + 1}},{\delta ^k}} \right)$. It follows immediately from \eqref{Eq.Doubling_dim_of_X_02} that
{
\setlength\abovedisplayskip{1ex}  
\setlength\belowdisplayskip{1ex}
\begin{align*}
\mu\left(Q_\alpha^k\right)^{-1} &\le V{\left( {z_x^k,{C_1}{\delta ^k}} \right)^{ - 1}} \\
 &\le C{\left( {1 + \frac{{d\left( {x,z_x^k} \right)}}{{{C_1}{\delta ^k}}}} \right)^m}V{\left( {x,{C_1}{\delta ^k}} \right)^{ - 1}}\\
 &\le CV{\left( {x,{C_1}{\delta ^k}} \right)^{ - 1}}\\
 &\le CV{\left( {x,{\delta ^k}} \right)^{ - 1}}\\
 &\le CV{\left( {x,t} \right)^{ - 1}},
\end{align*}
where the last but one inequality follows from the fact that $V\left(x, \delta^k\right)\le C_D C_1^{-m}V\left(x, C_1\delta^k\right)$ with $C_1 \in \left(0, 1\right)$. Hence, by this estimate and \eqref{Proof_of_Theorem_3_2_02}, we have
}
{
\setlength\abovedisplayskip{1ex}  
\setlength\belowdisplayskip{1ex}
\begin{align*}
{W_f}\left( x \right) &\le C{\left\{ {\sum\limits_{k \in \mathbb{Z}} {\int_{{\delta ^{k + 1}}}^{{\delta ^k}} {V{{\left( {x,t} \right)}^{ - 1}}\int_{B\left( {x,{C_2}{\delta ^{ - 1}}t} \right)} {|{t^2}L{e^{ - {t^2}L}}f\left( y \right){|^2}d\mu \left( y \right)\frac{{dt}}{t}} } } } \right\}^{{1 \mathord{\left/
 {\vphantom {1 2}} \right.
 \kern-\nulldelimiterspace} 2}}}\\
& = C{\left\{ {\int_0^\infty  {\int_{d\left( {x,y} \right) < {\delta ^{ - 2}}t} {|{t^2}L{e^{ - {t^2}L}}f\left( y \right){|^2}\frac{{d\mu \left( y \right)}}{{V\left( {x,t} \right)}}\frac{{dt}}{t}} } } \right\}^{{1 \mathord{\left/
 {\vphantom {1 2}} \right.
 \kern-\nulldelimiterspace} 2}}}\\
 &= C{S_{L,{\delta ^{ - 2}}}}\left( f \right)\left( x \right),
\end{align*}
which, together with Lemma \ref{Lemma_3_6}, we deduce the $\ge$ inequality of \eqref{Proof_of_Theorem_3_2_01}.

It remains to establish the reverse inequality. In the view of Proposition \ref{Pro.ATLM_Family}, we write
 $f = \sum\limits_{k \in \mathbb{Z}} {\sum\limits_{\alpha  \in {I_k}} {{s_{Q_\alpha ^k}}{a_{Q_\alpha ^k}}} }$. Let $\delta$ be as in Lemma \ref{Lemma_Decompositon_of_Space} we get
}
{
 \setlength\abovedisplayskip{1ex}  
 \setlength\belowdisplayskip{1ex}
 \begin{align}\label{Proof_of_Theorem_3_2_03}
&{S_L}\left( f \right)\left( x \right)\nonumber\\&
\begin{array}{*{20}{c}}
{}&{}
\end{array} = {\left( {\int_0^\infty  {\int_{d\left( {x,y} \right) < t} {|{t^2}L{e^{ - {t^2}L}}(f)\left( y \right){|^2}\frac{{d\mu \left( y \right)}}{{V\left( {x,t} \right)}}} \frac{{dt}}{t}} } \right)^{{1 \mathord{\left/
 {\vphantom {1 2}} \right.
 \kern-\nulldelimiterspace} 2}}}\nonumber\\&
\begin{array}{*{20}{c}}
{}&{}
\end{array} = {\left( {\int_0^\infty  {\int_{d\left( {x,y} \right) < t} {|{t^2}L{e^{ - {t^2}L}}(\sum\limits_{k \in \mathbb{Z}} {\sum\limits_{\alpha  \in {I_k}} {{s_{Q_\alpha ^k}}{a_{Q_\alpha ^k}}} } )\left( y \right){|^2}\frac{{d\mu \left( y \right)}}{{V\left( {x,t} \right)}}} \frac{{dt}}{t}} } \right)^{{1 \mathord{\left/
 {\vphantom {1 2}} \right.
 \kern-\nulldelimiterspace} 2}}}\nonumber\\&
\begin{array}{*{20}{c}}
{}&{}
\end{array} = {\left( {\sum\limits_{j \in \mathbb{Z}} {\int_{{\delta ^j}}^{{\delta ^{j - 1}}} {\int_{d\left( {x,y} \right) < t} {|{t^2}L{e^{ - {t^2}L}}(\sum\limits_{k \in \mathbb{Z}} {\sum\limits_{\alpha  \in {I_k}} {{s_{Q_\alpha ^k}}{a_{Q_\alpha ^k}}} } )\left( y \right){|^2}\frac{{d\mu \left( y \right)}}{{V\left( {x,t} \right)}}} \frac{{dt}}{t}} } } \right)^{{1 \mathord{\left/
 {\vphantom {1 2}} \right.
 \kern-\nulldelimiterspace} 2}}}\nonumber\\&
\begin{array}{*{20}{c}}
{}&{}
\end{array} \le {\left( {\sum\limits_{j \in \mathbb{Z}} {\int_{{\delta ^j}}^{{\delta ^{j - 1}}} {\int_{d\left( {x,y} \right) < t} {|{t^2}L{e^{ - {t^2}L}}(\sum\limits_{k > j} {\sum\limits_{\alpha  \in {I_k}} {{s_{Q_\alpha ^k}}{a_{Q_\alpha ^k}}} } )\left( y \right){|^2}\frac{{d\mu \left( y \right)}}{{V\left( {x,t} \right)}}} \frac{{dt}}{t}} } } \right)^{{1 \mathord{\left/
 {\vphantom {1 2}} \right.
 \kern-\nulldelimiterspace} 2}}}\nonumber\\&
\begin{array}{*{20}{c}}
{}&{}&{}
\end{array} + {\left( {\sum\limits_{j \in \mathbb{Z}} {\int_{{\delta ^j}}^{{\delta ^{j - 1}}} {\int_{d\left( {x,y} \right) < t} {|{t^2}L{e^{ - {t^2}L}}(\sum\limits_{k \le j} {\sum\limits_{\alpha  \in {I_k}} {{s_{Q_\alpha ^k}}{a_{Q_\alpha ^k}}} } )\left( y \right){|^2}\frac{{d\mu \left( y \right)}}{{V\left( {x,t} \right)}}} \frac{{dt}}{t}} } } \right)^{{1 \mathord{\left/
 {\vphantom {1 2}} \right.
 \kern-\nulldelimiterspace} 2}}}.
 \end{align}

 We now estimate the first part of \eqref{Proof_of_Theorem_3_2_03}. For any $k > j$ and $\alpha \in I_k$, noting that $a_{Q_{\alpha}^k} = L^Mb_{Q_{\alpha}^k}$, we write
 }
 $$
 \setlength\abovedisplayskip{1ex}  
 \setlength\belowdisplayskip{1ex}
 \left| {{t^2}L{e^{ - {t^2}L}}\left( {{a_{Q_\alpha ^k}}} \right)(y)} \right| = \left| {{t^2}{L^{M + 1}}{e^{ - {t^2}L}}\left( {{b_{Q_\alpha ^k}}} \right)(y)} \right| = {t^{ - 2M}}\left| {{{\left( {{t^2}L} \right)}^{M + 1}}{e^{ - {t^2}L}}\left( {{b_{Q_\alpha ^k}}} \right)(y)} \right|.
 $$
  Let $n$ be as in \eqref{Eq.Doubling_dim_of_X_01}, since $M > {{nq\left( \varphi  \right)} \mathord{\left/
 {\vphantom {{nq\left( \varphi  \right)} {\left( {2{p_1}} \right)}}} \right.
 \kern-\nulldelimiterspace} {\left( {2{p_1}} \right)}}$, we can choose some $q = r$ with $r$ be as in Corollary \ref{Corollary_Property_of_Growth_Functions_04} such that $2M > {n \mathord{\left/
 {\vphantom {n q}} \right.
 \kern-\nulldelimiterspace} q}$. We then let $N$ be some positive number such that $2M > N > {n \mathord{\left/
 {\vphantom {n q}} \right.
 \kern-\nulldelimiterspace} q}$. Then by Definition \ref{Def.AT_LM_Family}, the upper bound of the kernel of $\left(t^2L\right)^{M+1}e^{-t^2L}$ and \eqref{Eq.Doubling_Int_Properties}, we get
 {
 \setlength\abovedisplayskip{1ex}  
 \setlength\belowdisplayskip{1ex}
 \begin{align}
&\left| {{t^2}L{e^{ - {t^2}L}}\left( {{a_{Q_\alpha ^k}}} \right)(y)} \right|\nonumber\\&
\begin{array}{*{20}{c}}
{}&{}&{}
\end{array} \le \frac{{{C_5}}}{{V\left( {y,t} \right)}}{t^{ - 2M}}\ell {\left( {Q_\alpha ^k} \right)^{2M}}\mu {\left( {Q_\alpha ^k} \right)^{{{ - 1} \mathord{\left/
 {\vphantom {{ - 1} 2}} \right.
 \kern-\nulldelimiterspace} 2}}}\int_{3Q_\alpha ^k} {\exp \left( { - \frac{{d{{\left( {y,z} \right)}^2}}}{{{C_6}{t^2}}}} \right)} d\mu \left( z \right)\nonumber\\&
\begin{array}{*{20}{c}}
{}&{}&{}
\end{array} \le C{t^{ - 2M}}\ell {\left( {Q_\alpha ^k} \right)^{2M}}\mu {\left( {Q_\alpha ^k} \right)^{{{ - 1} \mathord{\left/
 {\vphantom {{ - 1} 2}} \right.
 \kern-\nulldelimiterspace} 2}}}{\left( {\frac{t}{{t + d(y,z_\alpha ^k)}}} \right)^N},\nonumber
 \end{align}
 where we denote by $z_{\alpha}^{k}$ the center of $Q_{\alpha}^k$. By the fact that $d(x, y) < t$, we further obtain
 }
 {
 \setlength\abovedisplayskip{1ex}  
 \setlength\belowdisplayskip{1ex}
 \begin{align*}
&{\left( {\int_{d\left( {x,y} \right) < t} {|{t^2}L{e^{ - {t^2}L}}({a_{Q_\alpha ^k}})\left( y \right){|^2}\frac{{d\mu \left( y \right)}}{{V\left( {x,t} \right)}}} } \right)^{{1 \mathord{\left/
 {\vphantom {1 2}} \right.
 \kern-\nulldelimiterspace} 2}}}\\&
\begin{array}{*{20}{c}}
{}&{}&{}
\end{array} \le C{t^{ - 2M}}\ell {\left( {Q_\alpha ^k} \right)^{2M}}\mu {\left( {Q_\alpha ^k} \right)^{{{ - 1} \mathord{\left/
 {\vphantom {{ - 1} 2}} \right.
 \kern-\nulldelimiterspace} 2}}}{\left( {\int_{d\left( {x,y} \right) < t} {{{\left( {\frac{{t + d\left( {x,y} \right)}}{{t + d\left( {x,z_\alpha ^k} \right)}}} \right)}^{2N}}\frac{{d\mu \left( y \right)}}{{V\left( {x,t} \right)}}} } \right)^{{1 \mathord{\left/
 {\vphantom {1 2}} \right.
 \kern-\nulldelimiterspace} 2}}}\\&
\begin{array}{*{20}{c}}
{}&{}&{}
\end{array} \le C{t^{ - 2M}}\ell {\left( {Q_\alpha ^k} \right)^{2M}}\mu {\left( {Q_\alpha ^k} \right)^{{{ - 1} \mathord{\left/
 {\vphantom {{ - 1} 2}} \right.
 \kern-\nulldelimiterspace} 2}}}{\left( {1 + {t^{ - 1}}d(x,z_\alpha ^k)} \right)^{ - N}}.
 \end{align*}
 Hence, we have
 }
 {
 \setlength\abovedisplayskip{1ex}  
 \setlength\belowdisplayskip{1ex}
 \begin{align}\label{Proof_of_Theorem_3_2_04}
&{\left( {\int_{d\left( {x,y} \right) < t} {|{t^2}L{e^{ - {t^2}L}}(\sum\limits_{k > j} {\sum\limits_{\alpha  \in {I_k}} {{s_{Q_\alpha ^k}}{a_{Q_\alpha ^k}}} } )\left( y \right){|^2}\frac{{d\mu \left( y \right)}}{{V\left( {x,t} \right)}}} } \right)^{{1 \mathord{\left/
 {\vphantom {1 2}} \right.
 \kern-\nulldelimiterspace} 2}}} \nonumber\\&
\begin{array}{*{20}{c}}
{}&{}&{}
\end{array} \le C\sum\limits_{k > j} {\sum\limits_{\alpha  \in {I_k}} {{t^{ - 2M}}\ell {{\left( {Q_\alpha ^k} \right)}^{2M}}\mu {{\left( {Q_\alpha ^k} \right)}^{{{ - 1} \mathord{\left/
 {\vphantom {{ - 1} 2}} \right.
 \kern-\nulldelimiterspace} 2}}}|{s_{Q_\alpha ^k}}|{{\left( {1 + {t^{ - 1}}d(x,z_\alpha ^k)} \right)}^{ - N}}} }  \nonumber\\&
\begin{array}{*{20}{c}}
{}&{}&{}
\end{array} \le C\sum\limits_{k > j} {{\delta ^{(2M - N)(k - j)}}\sum\limits_{\alpha  \in {I_k}} {\mu {{\left( {Q_\alpha ^k} \right)}^{{{ - 1} \mathord{\left/
 {\vphantom {{ - 1} 2}} \right.
 \kern-\nulldelimiterspace} 2}}}\frac{{|{s_{Q_\alpha ^k}}|}}{{{{\left[ {1 + \ell {{( {Q_\alpha ^k} )}^{ - 1}}d(x,z_\alpha ^k)} \right]}^N}}}} } \nonumber\\&
\begin{array}{*{20}{c}}
{}&{}&{}
\end{array} \le C\sum\limits_{k > j} {{\delta ^{\left( {2M - N} \right)\left( {k - j} \right)}}{{\left[ {{\cal M}\left( {\sum\limits_{\alpha  \in {I_k}} {|{s_{Q_\alpha ^k}}{|^q}\mu {{\left( {Q_\alpha ^k} \right)}^{{{ - q} \mathord{\left/
 {\vphantom {{ - q} 2}} \right.
 \kern-\nulldelimiterspace} 2}}}{\mathcal{X}_{Q_\alpha ^k}}\left(  \cdot  \right)} } \right)\left( x \right)} \right]}^{{1 \mathord{\left/
 {\vphantom {1 q}} \right.
 \kern-\nulldelimiterspace} q}}}} ,
 \end{align}
 where the last inequality follows from Lemma \ref{Lemma_3_2}.

 Estimate of the second part of \eqref{Proof_of_Theorem_3_2_03}. For any $k \le j$ and $\alpha \in I_k$, we write
 }
 $$
 \setlength\abovedisplayskip{1ex}  
 \setlength\belowdisplayskip{1ex}
 \left| {{t^2}L{e^{ - {t^2}L}}\left( {{a_{Q_\alpha ^k}}} \right)(y)} \right| = {t^2}\left| {{e^{ - {t^2}L}}\left( {L\left( {{a_{Q_\alpha ^k}}} \right)} \right)(y)} \right|.
 $$
 Then by Definition \ref{Def.AT_LM_Family}, the Gaussian estimate \eqref{Eq.GE} and inequality \eqref{Eq.Doubling_Int_Properties}, we get
 {
 \setlength\abovedisplayskip{1ex}  
 \setlength\belowdisplayskip{1ex}
 \begin{align}
&\left| {{t^2}L{e^{ - {t^2}L}}\left( {{a_{Q_\alpha ^k}}} \right)(y)} \right|\nonumber\\
&\begin{array}{*{20}{c}}
{}&{}&{}
\end{array} \le \frac{{{C_5}}}{{V\left( {y,t} \right)}}{t^2}\ell {\left( {Q_\alpha ^k} \right)^{ - 2}}\mu {\left( {Q_\alpha ^k} \right)^{{{ - 1} \mathord{\left/
 {\vphantom {{ - 1} 2}} \right.
 \kern-\nulldelimiterspace} 2}}}\int_{3Q_\alpha ^k} {\exp \left( { - \frac{{d{{\left( {y,z} \right)}^2}}}{{{C_6}{t^2}}}} \right)d\mu \left( z \right)}.\nonumber\\
&\begin{array}{*{20}{c}}
{}&{}&{}
\end{array} \le C{t^2}\ell {\left( {Q_\alpha ^k} \right)^{ - 2}}\mu {\left( {Q_\alpha ^k} \right)^{{{ - 1} \mathord{\left/
 {\vphantom {{ - 1} 2}} \right.
 \kern-\nulldelimiterspace} 2}}}{\left( {1 + \ell {{(Q_\alpha ^k)}^{ - 1}}d( {y,z_\alpha ^k})} \right)^{ - N}},\nonumber
 \end{align}
 which together with the fact that $d\left(x, y\right) < t \le \ell(Q_{\alpha}^k)$ further implies that
 }
 {
 \setlength\abovedisplayskip{1ex}  
 \setlength\belowdisplayskip{1ex}
 \begin{align*}
&{\left( {\int_{d\left( {x,y} \right) < t} {|{t^2}L{e^{ - {t^2}L}}({a_{Q_\alpha ^k}})\left( y \right){|^2}\frac{{d\mu \left( y \right)}}{{V\left( {x,t} \right)}}} } \right)^{{1 \mathord{\left/
 {\vphantom {1 2}} \right.
 \kern-\nulldelimiterspace} 2}}}\\&
\begin{array}{*{20}{c}}
{}&{}&{}
\end{array} \le C{t^2}\ell {\left( {Q_\alpha ^k} \right)^{ - 2}}\mu {\left( {Q_\alpha ^k} \right)^{{{ - 1} \mathord{\left/
 {\vphantom {{ - 1} 2}} \right.
 \kern-\nulldelimiterspace} 2}}}{\left( {\int_{d\left( {x,y} \right) < t} {{{\left( {\frac{{\ell ( {Q_\alpha ^k} ) + d( {x,y})}}{{\ell ( {Q_\alpha ^k}) + d( {x,z_\alpha ^k} )}}} \right)}^{2N}}\frac{{d\mu \left( y \right)}}{{V\left( {x,t} \right)}}} } \right)^{{1 \mathord{\left/
 {\vphantom {1 2}} \right.
 \kern-\nulldelimiterspace} 2}}}\\&
\begin{array}{*{20}{c}}
{}&{}&{}
\end{array} \le C{t^2}\ell {\left( {Q_\alpha ^k} \right)^{ - 2}}\mu {\left( {Q_\alpha ^k} \right)^{{{ - 1} \mathord{\left/
 {\vphantom {{ - 1} 2}} \right.
 \kern-\nulldelimiterspace} 2}}}{\left( {1 + \ell {{( {Q_\alpha ^k} )}^{ - 1}}d(x,z_\alpha ^k)} \right)^{ - N}}.
 \end{align*}
 Therefore, by employing Lemma \ref{Lemma_3_2} once again, we get
 }
 {
 \setlength\abovedisplayskip{1ex}  
 \setlength\belowdisplayskip{1ex}
 \begin{align}\label{Proof_of_Theorem_3_2_05}
&{\left( {\int_{d\left( {x,y} \right) < t} {|{t^2}L{e^{ - {t^2}L}}(\sum\limits_{k \le j} {\sum\limits_{\alpha  \in {I_k}} {{s_{Q_\alpha ^k}}{a_{Q_\alpha ^k}}} } )\left( y \right){|^2}\frac{{d\mu \left( y \right)}}{{V\left( {x,t} \right)}}} } \right)^{{1 \mathord{\left/
 {\vphantom {1 2}} \right.
 \kern-\nulldelimiterspace} 2}}}\nonumber\\&
\begin{array}{*{20}{c}}
{}&{}&{}
\end{array} \le C\sum\limits_{k \le j} {\sum\limits_{\alpha  \in {I_k}} {{t^2}\ell {{\left( {Q_\alpha ^k} \right)}^{ - 2}}\mu {{\left( {Q_\alpha ^k} \right)}^{{{ - 1} \mathord{\left/
 {\vphantom {{ - 1} 2}} \right.
 \kern-\nulldelimiterspace} 2}}}|{s_{Q_\alpha ^k}}|{{\left( {1 + \ell {{(Q_\alpha ^k)}^{ - 1}}d(x,z_\alpha ^k)} \right)}^{ - N}}} } \nonumber\\&
\begin{array}{*{20}{c}}
{}&{}&{}
\end{array} \le C\sum\limits_{k \le j} {{\delta ^{2(j - k)}}\sum\limits_{\alpha  \in {I_k}} {\mu {{\left( {Q_\alpha ^k} \right)}^{{{ - 1} \mathord{\left/
 {\vphantom {{ - 1} 2}} \right.
 \kern-\nulldelimiterspace} 2}}}\frac{{|{s_{Q_\alpha ^k}}|}}{{{{\left[ {1 + \ell {{(Q_\alpha ^k)}^{ - 1}}d(x,z_\alpha ^k)} \right]}^N}}}} } \nonumber\\&
\begin{array}{*{20}{c}}
{}&{}&{}
\end{array} \le C\sum\limits_{k \le j} {{\delta ^{2(j - k)}}{{\left[ {{\cal M}\left( {\sum\limits_{\alpha  \in {I_k}} {|{s_{Q_\alpha ^k}}{|^q}\mu {{\left( {Q_\alpha ^k} \right)}^{{{ - q} \mathord{\left/
 {\vphantom {{ - q} 2}} \right.
 \kern-\nulldelimiterspace} 2}}}{\mathcal{X}_{Q_\alpha ^k}}\left(  \cdot  \right)} } \right)\left( x \right)} \right]}^{{1 \mathord{\left/
 {\vphantom {1 q}} \right.
 \kern-\nulldelimiterspace} q}}}}.
 \end{align}
 Set $F_k\left(x\right):={\cal M}\left( {\sum\nolimits_{\alpha  \in {I_k}} {|{s_{Q_\alpha ^k}}{|^q}\mu {{(Q_\alpha ^k)}^{{{ - q} \mathord{\left/
 {\vphantom {{ - q} 2}} \right.
 \kern-\nulldelimiterspace} 2}}}{\mathcal{X}_{Q_\alpha ^k}}( \cdot )} } \right)\left( x \right)$. For any fix $j \in \mathbb{Z}$, we let $\beta > 0$, $\tau = 1$ if $k > j$ and $\tau = -1$ if $k \le j$. We now turn to estimate ${\left( {\sum\nolimits_k {{\delta ^{\beta \tau (k - j)}}{F_k}{{\left( x \right)}^{{1 \mathord{\left/
 {\vphantom {1 q}} \right.
 \kern-\nulldelimiterspace} q}}}} } \right)^2}$ for the case either $k > j$ or $k \le j$. Since
 }
 $$
 \setlength\abovedisplayskip{1ex}  
 \setlength\belowdisplayskip{1ex}
 {\delta ^{\beta \tau \left( {k - j} \right)}} = \frac{{\beta {\delta ^\beta }}}{{1 - {\delta ^\beta }}}\int_{{\delta ^{\tau \left( {k - j} \right)}}}^{{\delta ^{\tau \left( {k - j} \right) - 1}}} {{s^{\beta  - 1}}ds},
 $$
 we let $E_k:=\left[{{\delta ^{\tau \left( {k - j} \right)}}}, {{\delta ^{\tau \left( {k - j} \right)-1}} }\right]$ and it follows that
 {
 \setlength\abovedisplayskip{1ex}  
 \setlength\belowdisplayskip{1ex}
 \begin{align*}
{\left( {\sum\nolimits_k {{\delta ^{\beta \tau (k - j)}}{F_k}{{\left( x \right)}^{{1 \mathord{\left/
 {\vphantom {1 q}} \right.
 \kern-\nulldelimiterspace} q}}}} } \right)^2} &= {C_\beta }{\left( {\sum\limits_k {\int_{{\delta ^{\tau \left( {k - j} \right)}}}^{{\delta ^{\tau \left( {k - j} \right) - 1}}} {{F_k}{{\left( x \right)}^{{1 \mathord{\left/
 {\vphantom {1 q}} \right.
 \kern-\nulldelimiterspace} q}}}{s^{\beta  - 1}}ds} } } \right)^2}\nonumber\\
& = {C_\beta }{\left( {\int_0^1 {\sum\limits_k {{\chi _{{E_k}}}(s){F_k}{{\left( x \right)}^{{1 \mathord{\left/
 {\vphantom {1 q}} \right.
 \kern-\nulldelimiterspace} q}}}{s^{\beta  - 1}}} ds} } \right)^2}\nonumber\\
& \le {C_\beta }\left( {\int_0^1 {{s^{\beta  - 1}}ds} } \right)\left( {\int_0^1 {{{\left( {\sum\nolimits_k {{\chi _{{E_k}}}(s){F_k}{{\left( x \right)}^{{1 \mathord{\left/
 {\vphantom {1 q}} \right.
 \kern-\nulldelimiterspace} q}}}} } \right)}^2}{s^{\beta  - 1}}ds} } \right)\nonumber\\
& \le {C_\beta }\int_0^1 {{{\left( {\sum\nolimits_k {{\chi _{{E_k}}}(s){F_k}{{\left( x \right)}^{{1 \mathord{\left/
 {\vphantom {1 q}} \right.
 \kern-\nulldelimiterspace} q}}}} } \right)}^2}{s^{\beta  - 1}}ds} \nonumber\\
& = {C_\beta }\sum\limits_k {\int_{{\delta ^{\tau \left( {k - j} \right)}}}^{{\delta ^{\tau \left( {k - j} \right) - 1}}} {{F_k}{{\left( x \right)}^{{2 \mathord{\left/
 {\vphantom {2 q}} \right.
 \kern-\nulldelimiterspace} q}}}{s^{\beta  - 1}}ds} } \nonumber\\
\end{align*}
 \begin{align}\label{Proof_of_Theorem_3_2_06}
 & = {C_\beta }\sum\nolimits_k {{\delta ^{\beta \tau (k - j)}}{F_k}{{\left( x \right)}^{{2 \mathord{\left/
 {\vphantom {2 q}} \right.
 \kern-\nulldelimiterspace} q}}}}.
 \end{align}
 Combining now \eqref{Proof_of_Theorem_3_2_03}-\eqref{Proof_of_Theorem_3_2_06}, we get the following estimate for $S_L\left(f\right)$,
 }
{
 \setlength\abovedisplayskip{1ex}  
 \setlength\belowdisplayskip{1ex}
\begin{align*}
{S_L}\left( f \right)\left( x \right) &\le {\left( {\sum\limits_{j \in \mathbb{Z}} {\int_{{\delta ^j}}^{{\delta ^{j - 1}}} {\int_{d\left( {x,y} \right) < t} {|{t^2}L{e^{ - {t^2}L}}(\sum\limits_{k > j} {\sum\limits_{\alpha  \in {I_k}} {{s_{Q_\alpha ^k}}{a_{Q_\alpha ^k}}} } )\left( x \right){|^2}\frac{{d\mu \left( y \right)}}{{V\left( {x,t} \right)}}} \frac{{dt}}{t}} } } \right)^{{1 \mathord{\left/
 {\vphantom {1 2}} \right.
 \kern-\nulldelimiterspace} 2}}}\\&
\begin{array}{*{20}{c}}
{}&{}&{}
\end{array} + {\left( {\sum\limits_{j \in \mathbb{Z}} {\int_{{\delta ^j}}^{{\delta ^{j - 1}}} {\int_{d\left( {x,y} \right) < t} {|{t^2}L{e^{ - {t^2}L}}(\sum\limits_{k \le j} {\sum\limits_{\alpha  \in {I_k}} {{s_{Q_\alpha ^k}}{a_{Q_\alpha ^k}}} } )\left( x \right){|^2}\frac{{d\mu \left( y \right)}}{{V\left( {x,t} \right)}}} \frac{{dt}}{t}} } } \right)^{{1 \mathord{\left/
 {\vphantom {1 2}} \right.
 \kern-\nulldelimiterspace} 2}}}\\
 &\le C\left( {\sum\limits_{j \in \mathbb{Z}} {\int_{{\delta ^j}}^{{\delta ^{j - 1}}} {|\sum\limits_{k > j} {{\delta ^{\left( {2M - N} \right)\left( {k - j} \right)}}{F_k}{{\left( x \right)}^{{1 \mathord{\left/
 {\vphantom {1 q}} \right.
 \kern-\nulldelimiterspace} q}}}} {|^2}} \frac{{dt}}{t}} } \right.\\&
\begin{array}{*{20}{c}}
{}&{}&{}&{}
\end{array}{\left. { + \sum\limits_{j \in \mathbb{Z}} {\int_{{\delta ^j}}^{{\delta ^{j - 1}}} {|\sum\limits_{k \le j} {{\delta ^{2\left( {j - k} \right)}}{F_k}{{\left( x \right)}^{{1 \mathord{\left/
 {\vphantom {1 q}} \right.
 \kern-\nulldelimiterspace} q}}}} {|^2}} \frac{{dt}}{t}} } \right)^{{1 \mathord{\left/
 {\vphantom {1 2}} \right.
 \kern-\nulldelimiterspace} 2}}}\\
 &\le C\left( {\sum\limits_{j \in \mathbb{Z}} {\int_{{\delta ^j}}^{{\delta ^{j - 1}}} {\sum\limits_{k > j} {{\delta ^{\left( {2M - N} \right)\left( {k - j} \right)}}{F_k}{{\left( x \right)}^{{2 \mathord{\left/
 {\vphantom {2 q}} \right.
 \kern-\nulldelimiterspace} q}}}} } \frac{{dt}}{t}} } \right.\\&
\begin{array}{*{20}{c}}
{}&{}&{}&{}
\end{array}{\left. { + \sum\limits_{j \in \mathbb{Z}} {\int_{{\delta ^j}}^{{\delta ^{j - 1}}} {\sum\limits_{k \le j} {{\delta ^{2\left( {j - k} \right)}}{F_k}{{\left( x \right)}^{{2 \mathord{\left/
 {\vphantom {2 q}} \right.
 \kern-\nulldelimiterspace} q}}}} } \frac{{dt}}{t}} } \right)^{{1 \mathord{\left/
 {\vphantom {1 2}} \right.
 \kern-\nulldelimiterspace} 2}}}\\
& = C{\left( {\sum\limits_{k \in \mathbb{Z}} {{F_k}{{\left( x \right)}^{{2 \mathord{\left/
 {\vphantom {2 q}} \right.
 \kern-\nulldelimiterspace} q}}}(\sum\limits_{j > k} {{\delta ^{\left( {2M - N} \right)\left( {k - j} \right)}}}  + \sum\limits_{j \ge k} {{\delta ^{2\left( {j - k} \right)}}} )} } \right)^{{1 \mathord{\left/
 {\vphantom {1 2}} \right.
 \kern-\nulldelimiterspace} 2}}}\\
&\le C{\left( {\sum\limits_{k \in \mathbb{Z}} {{F_k}{{\left( x \right)}^{{2 \mathord{\left/
 {\vphantom {2 q}} \right.
 \kern-\nulldelimiterspace} q}}}} } \right)^{{1 \mathord{\left/
 {\vphantom {1 2}} \right.
 \kern-\nulldelimiterspace} 2}}}.
\end{align*}
where we used \eqref{Proof_of_Theorem_3_2_06} with $\beta = 2M-N$, $\tau = 1$ in the first sum and \eqref{Proof_of_Theorem_3_2_06} with $\beta = 2$, $\tau = -1$ in the second sum. Using this bound, we apply Corollary \ref{Corollary_Property_of_Growth_Functions_04} to get,
}
{
 \setlength\abovedisplayskip{1ex}  
 \setlength\belowdisplayskip{1ex}
\begin{align*}
&\int_X {\varphi \left( {x,{{{S_L}\left( f \right)(x)} \mathord{\left/
 {\vphantom {{{G_L}\left( f \right)(x)} \lambda }} \right.
 \kern-\nulldelimiterspace} \lambda }} \right)d\mu \left( x \right)} \\&
\begin{array}{*{20}{c}}
{}&{}&{}&{}
\end{array} \le C\int_X {\varphi \left( {x,{\lambda ^{ - 1}}{{\left( {\sum\nolimits_{k \in \mathbb{Z}} {{F_k}{{\left( x \right)}^{{2 \mathord{\left/
 {\vphantom {2 q}} \right.
 \kern-\nulldelimiterspace} q}}}} } \right)}^{{1 \mathord{\left/
 {\vphantom {1 2}} \right.
 \kern-\nulldelimiterspace} 2}}}} \right)d\mu \left( x \right)} \\&
\begin{array}{*{20}{c}}
{}&{}&{}&{}
\end{array} \le C\int_X {\varphi \left( {x,{\lambda ^{ - 1}}{{\left( {\sum\nolimits_{k \in \mathbb{Z}} {{{\left( {\sum\nolimits_{\alpha  \in {I_k}} {|{s_{Q_\alpha ^k}}{|^q}\mu {{\left( {Q_\alpha ^k} \right)}^{{{ - q} \mathord{\left/
 {\vphantom {{ - q} 2}} \right.
 \kern-\nulldelimiterspace} 2}}}{{\cal X}_{Q_\alpha ^k}}\left( x \right)} } \right)}^{{2 \mathord{\left/
 {\vphantom {2 q}} \right.
 \kern-\nulldelimiterspace} q}}}} } \right)}^{{1 \mathord{\left/
 {\vphantom {1 2}} \right.
 \kern-\nulldelimiterspace} 2}}}} \right)d\mu \left( x \right)} \\&
\begin{array}{*{20}{c}}
{}&{}&{}&{}
\end{array} = C\int_X {\varphi \left( {x,{\lambda ^{ - 1}}{{\left( {\sum\nolimits_{k \in \mathbb{Z}} {\sum\nolimits_{\alpha  \in {I_k}} {|{s_{Q_\alpha ^k}}{|^2}\mu {{\left( {Q_\alpha ^k} \right)}^{ - 1}}{{\cal X}_{Q_\alpha ^k}}\left( x \right)} } } \right)}^{{1 \mathord{\left/
 {\vphantom {1 2}} \right.
 \kern-\nulldelimiterspace} 2}}}} \right)d\mu \left( x \right)} \\&
\begin{array}{*{20}{c}}
{}&{}&{}&{}
\end{array} = C\int_X {\varphi \left( {x,{{{W_f}\left( x \right)} \mathord{\left/
 {\vphantom {{{W_f}\left( x \right)} \lambda }} \right.
 \kern-\nulldelimiterspace} \lambda }} \right)d\mu \left( x \right)},
\end{align*}
which proves the $\le$ inequality in \eqref{Proof_of_Theorem_3_2_01} and completes the proof of the theorem.
}

\end{proofoftheorem}

We now turn to characterize the Musielark-Orlicz Hardy space ${H_{L,G,\varphi }}$ and have the following result.
\begin{theorem}{}\label{Theorem_3_1}
    Suppose $L$ is an operator that satisfies \rm{(\textbf{H1})} \emph{and} \rm{(\textbf{H2})}. \emph{Let $\varphi$ be a growth function with uniformly lower type $p_1$ and $f \in {H_{L,G,\varphi }}\left( X \right) \cap {L^2}\left( X \right)$, then for all natural number $M > {{nq\left( \varphi  \right)} \mathord{\left/
 {\vphantom {{nq\left( \varphi  \right)} {\left( {2{p_1}} \right)}}} \right.
 \kern-\nulldelimiterspace} {\left( {2{p_1}} \right)}}$, $f$ has an \textbf{AT}$_{\mathcal{L}, M}$-expansion such that}
    $$
    \setlength\abovedisplayskip{1ex}  
    \setlength\belowdisplayskip{1ex}
    {\left\| f \right\|_{{H_{L,G,\varphi }\left(X\right)}}} \cong {\left\| {{W_f}} \right\|_{{L^\varphi \left(X\right)}}}.$$
\end{theorem}

We prove Theorem \ref{Theorem_3_1} by borrowing some ideas from \citet[Proof of Theorem 3.2]{Duong:2016}. To this end, we start with listing some known facts as follows.

Given $f \in L^2 \left( X \right)$, $a > 0$ and $\left( x, t\right) \in X \times \left( {0,\infty } \right)$, the Fefferman-Stein-type maximal function is defined as
$$
\setlength\abovedisplayskip{1ex}  
\setlength\belowdisplayskip{1ex}
M_{a,L}^*\left( f \right)\left( {x,t} \right) = \mathop {{\rm{esssup}}}\limits_{y \in X} \frac{{|{t^2}L{e^{ - {t^2}L}}f(y)|}}{{{{\left[ {1 + {t^{ - 1}}d\left( {x,y} \right)} \right]}^a}}},
$$
and we have the following Lemma \ref{Lemma_3_1} which was established in \citep{Guorong:2016}.
\begin{lemma}{}\label{Lemma_3_1}
Assume that L satisfies \rm{(\textbf{H1})} and \rm{(\textbf{H2})}. \emph{Let $m$ be as in \eqref{Eq.Doubling_dim_of_X_02}. Then, for any $\beta > 0$, $r > 0$ and $a > {m \mathord{\left/
 {\vphantom {m 2}} \right.
 \kern-\nulldelimiterspace} 2}$, there exist a positive constant $C$ such that for all $f \in L^2 \left( X \right)$, $l \in \mathbb{Z}$, $x \in X$ and $t \in \left[ 1, 2\right)$,}
 $$
 \setlength\abovedisplayskip{1ex}  
 \setlength\belowdisplayskip{1ex}
{\left| {M_{a,L}^*\left( f \right)\left( {x,{2^{ - l}}t} \right)} \right|^r} \le C\sum\limits_{j = l}^\infty  {{2^{ - \left( {j - l} \right)\beta r}}\int_X {\frac{{|{{({2^{ - j}}t)}^2}L{e^{ - {{({2^{ - j}}t)}^2}L}}f(z){|^r}}}{{V(z,{2^{ - l}}){{[1 + {2^l}d(x,z)]}^{ar}}}}d\mu \left( z \right)} } .$$
\end{lemma}

Moreover, we also need the following Lemma, whose proof is standard, we omit the details. And in what follows, we recall that the Hardy-Littlewood maximal operator $\mathcal{M}$ on $\left( X, \mu, d\right)$ is defined by
$$
\setlength\abovedisplayskip{1ex}
\setlength\belowdisplayskip{1ex}
\mathcal{M}\left( f \right)\left( x\right): = \mathop {\sup }\limits_{B \ni x} \frac{1}{{\mu\left( B \right)}}\int_B {\left| {f\left( y \right)} \right|d\mu \left( y \right)},
$$
where the supremum is taken over all balls $B \ni x$.
\begin{lemma}{}\label{Lemma_3_3}
Let $n$ and $m$ be as in \eqref{Eq.Doubling_dim_of_X_01} and \eqref{Eq.Doubling_dim_of_X_02}, and suppose that $N > n + m$. Then there exists a constant $C > 0$ such that for all measurable functions $f$ on $\left( X, \mu, d\right)$, $t > 0$ and each $y \in X$,
$$
 \setlength\abovedisplayskip{1ex}
 \setlength\belowdisplayskip{1ex}
\int_X {\frac{{\left| {f\left( x \right)} \right|}}{{V\left( {x,t} \right){{\left[ {1 + {t^{ - 1}}d\left( {x,y} \right)} \right]}^N}}}d\mu \left( x \right)}  \le C\mathcal{M}\left( f \right)\left( y \right) .
$$
\end{lemma}

\begin{proofoftheorem}{\ref{Theorem_3_1}}
  For any fixed $f \in {H_{L,G,\varphi }}\left( X \right) \cap {L^2}\left( X \right)$, we let ${\lambda _0} = {\left\| f \right\|_{{H_{L,G,\varphi }}\left(X\right)}}$ and ${\lambda _1} = {\left\| {{W_f}} \right\|_{{L^\varphi }\left(X\right)}}$. It suffices to show that for all $\lambda \in \left(0, \infty \right)$,
  {
  \setlength\abovedisplayskip{1ex}  
  \setlength\belowdisplayskip{1ex}
  \begin{equation}\label{Proof_of_Theorem_3_1_01}
    \int_X {\varphi \left( {x,\frac{{{G_L}\left( f \right)\left( x \right)}}{\lambda }} \right)d\mu \left( x \right)}  \cong \int_X {\varphi \left( {x,\frac{{\left| {{W_f}\left( x \right)} \right|}}{\lambda }} \right)d\mu \left( x \right)}.
  \end{equation}
  In fact, if \eqref{Proof_of_Theorem_3_1_01} holds for all $\lambda \in \left(0, \infty \right)$, then there exists a constant $C_0$ such that
  }
  $$
  \setlength\abovedisplayskip{1ex}  
  \setlength\belowdisplayskip{1ex}
  \int_X {\varphi \left( {x,\frac{{{G_L}\left( f \right)\left( x \right)}}{\lambda _1 }} \right)d\mu \left( x \right)}  \le {C_0}\int_X {\varphi \left( {x,\frac{{\left| {{W_f}\left( x \right)} \right|}}{{{\lambda _1}}}} \right)d\mu \left( x \right)}  \le {C_0} ,
  $$
  which, together with \eqref{Eq.Property_of_Growth_Function}, implies that
  $$
  \setlength\abovedisplayskip{1ex}  
  \setlength\belowdisplayskip{1ex}
  \int_X {\varphi \left( {x,\frac{{{G_L}\left( f \right)\left( x \right)}}{{{C_1}{\lambda _1}}}} \right)d\mu \left( x \right)}  \le 1
  $$
  for some constants $C_1$, and hence we have $\lambda _0 \le C_1 \lambda _1$. In a similar fashion, one can prove $\lambda _1 \le C_2 \lambda _0$ for some constants $C_2$, and get the desired property.

  Now we fix arbitrary $\lambda \in \left(0, \infty\right)$ and prove \eqref{Proof_of_Theorem_3_1_01}. In the view of Lemma \ref{Lemma_Decompositon_of_Space}, for any fixed $\left( x, k\right) \in X \times \mathbb{Z}$ there exists a unique $\alpha \in I_{k}$, such that $x \in Q_\alpha ^k$. Let $Q_x ^k$ denote the such $Q_\alpha ^k$ and we write
  {
  \setlength\abovedisplayskip{1ex}  
  \setlength\belowdisplayskip{1ex}
\begin{align}\label{Eq.W_F_01}
{W_f}\left( x \right)
 &= {\left\{ {\sum\limits_{k \in \mathbb{Z}} {\sum\limits_{\alpha  \in {I_k}} {{{\left[ {\mu{{\left( {Q_\alpha ^k} \right)}^{ - {1 \mathord{\left/
 {\vphantom {1 2}} \right.
 \kern-\nulldelimiterspace} 2}}}\left| {{s_{Q_\alpha ^k}}} \right|{\mathcal{X} _{Q_\alpha ^k}}\left( x \right)} \right]}^2}} } } \right\}^{{1 \mathord{\left/
 {\vphantom {1 2}} \right.
 \kern-\nulldelimiterspace} 2}}} \nonumber\\
 &= {\left\{ {\sum\limits_{k \in \mathbb{Z}} {\mu{{\left( {Q_x^k} \right)}^{ - 1}}{{\left| {{s_{Q_x^k}}} \right|}^2}} } \right\}^{{1 \mathord{\left/
 {\vphantom {1 2}} \right.
 \kern-\nulldelimiterspace} 2}}}\nonumber\\
 &= {\left\{ {\sum\limits_{k \in \mathbb{Z}} {\int_{{\delta ^{k + 1}}}^{{\delta ^k}} {\mu{{\left( {Q_x^k} \right)}^{ - 1}}\int_{Q_x^k} {{{\left| {{t^2}L{e^{ - {t^2}L}}f\left( y \right)} \right|}^2}d\mu \left( y \right)\frac{{dt}}{t}} } } } \right\}^{{1 \mathord{\left/
 {\vphantom {1 2}} \right.
 \kern-\nulldelimiterspace} 2}}},
\end{align}
where $\delta \in \left( 0, 1\right)$ is a constant as in Lemma \ref{Lemma_Decompositon_of_Space} and the last quantity follows from Proposition \ref{Pro.ATLM_Family}. Moreover, by $\left( {iv} \right)$ and $\left( {v} \right)$ of Lemma \ref{Lemma_Decompositon_of_Space}, we know that for any fixed $\left( x, k\right) \in X \times \mathbb{Z}$ there exists $z_x^k \in Q_x^k$ and constants $C_3 \in \left( 0, 1\right)$, $C_4 > 0$ such that ${\rm{diam}}\left( {Q_\alpha ^k} \right) \le {C_4}{\delta ^k}$ and
}
$$
\setlength\abovedisplayskip{1ex}  
\setlength\belowdisplayskip{1ex}
B\left( {z_x^k,{C_3}{\delta ^k}} \right) \subset Q_x^k \subset B\left( {x,{C_4}{\delta ^k}} \right) \subset B\left( {x,{C_4}{\delta ^{ - 1}}t} \right): = {B_x}
$$
for all $t \in \left( {{\delta ^{k + 1}},{\delta ^k}} \right)$. Then for each $k \in \mathbb{Z}$ we compute by \eqref{Eq.Doubling_dim_of_X_01} and \eqref{Eq.Doubling_dim_of_X_02},
{
\setlength\abovedisplayskip{1ex}  
\setlength\belowdisplayskip{1ex}
\begin{align}\label{Eq.W_F_02}
&\int_{{\delta ^{k + 1}}}^{{\delta ^k}} {\mu {{\left( {Q_x^k} \right)}^{ - 1}}\int_{Q_x^k} {{{\left| {{t^2}L{e^{ - {t^2}L}}f\left( y \right)} \right|}^2}d\mu \left( y \right)\frac{{dt}}{t}} } \nonumber\\&
\begin{array}{*{20}{c}}
{}&{}&{}&{}
\end{array}\le \int_{{\delta ^{k + 1}}}^{{\delta ^k}} {\mu {{\left( {B\left( {z_x^k,{C_3}{\delta ^k}} \right)} \right)}^{ - 1}}\int_{B(x,{C_4}{\delta ^k})} {{{\left| {{t^2}L{e^{ - {t^2}L}}f\left( y \right)} \right|}^2}d\mu \left( y \right)\frac{{dt}}{t}} } \nonumber\\&
\begin{array}{*{20}{c}}
{}&{}&{}&{}
\end{array}\le C\int_{{\delta ^{k + 1}}}^{{\delta ^k}} {\frac{{\mu \left( {B\left( {x,{C_4}{\delta ^k}} \right)} \right)}}{{\mu \left( {B\left( {z_x^k,{C_3}{\delta ^k}} \right)} \right)}}\mathop {{\rm{esssup}}}\limits_{y \in B(x,{C_4}{\delta ^k})} {{\left| {{t^2}L{e^{ - {t^2}L}}f\left( y \right)} \right|}^2}\frac{{dt}}{t}} \nonumber\\&
\begin{array}{*{20}{c}}
{}&{}&{}&{}
\end{array}\le C\int_{{\delta ^{k + 1}}}^{{\delta ^k}} {\mathop {{\rm{esssup}}}\limits_{y \in {B_x}} {{\left| {{t^2}L{e^{ - {t^2}L}}f\left( y \right)} \right|}^2}\frac{{dt}}{t}} \nonumber\\&
\begin{array}{*{20}{c}}
{}&{}&{}&{}
\end{array}\le C\int_{{\delta ^{k + 1}}}^{{\delta ^k}} {{{\left[ {M_{a,L}^*\left( f \right)\left( {x,t} \right)} \right]}^2}\frac{{dt}}{t}}
\end{align}
for some appropriate constant $C$, where $M_{a ,L}^ * \left( f \right)\left( {x,t} \right)$ is the Fefferman-Stein-type maximal function, with some large enough constants $a$ to be selected. And the last inequality follows from
}\\
{
\setlength\abovedisplayskip{1ex}  
\setlength\belowdisplayskip{1ex}
\begin{align*}
\mathop {{\rm{esssup}}}\limits_{y \in {B_x}} {\left| {{t^2}L{e^{ - {t^2}L}}f\left( y \right)} \right|^2}
&= \mathop {{\rm{esssup}}}\limits_{y \in {B_x}} \frac{{|{t^2}L{e^{ - {t^2}L}}f\left( y \right){|^2}}}{{{{[1 + {t^{ - 1}}d\left( {x,y} \right)]}^{2a}}}}{\left[ {1 + {t^{ - 1}}d\left( {x,y} \right)} \right]^{2a}}\\
&\le \left( {1 + {C_1}{\delta ^{ - 1}}} \right){\left[ {M_{a,L}^*\left( f \right)\left( {x,t} \right)} \right]^2}.
\end{align*}
Combining now \eqref{Eq.W_F_01}-\eqref{Eq.W_F_02}, we have the following estimate of $W_{f}\left(x\right)$,
}
{
\setlength\abovedisplayskip{1ex}  
\setlength\belowdisplayskip{1ex}
\begin{align*}
    {W_f}\left( x \right)
    &\le C{\left\{ {\int_0^\infty  {{{\left[ {M_{a,L}^ * \left( f \right)\left( {x,t} \right)} \right]}^2}\frac{{dt}}{t}} } \right\}^{{1 \mathord{\left/
 {\vphantom {1 2}} \right.
 \kern-\nulldelimiterspace} 2}}}\\
    &= C{\left\{ {\sum\limits_{k \in \mathbb{Z}} {\int_{{2^{-k}}}^{{2^{-k + 1}}} {{{\left[ {M_{a,L}^ * \left( f \right)\left( {x,t} \right)} \right]}^2}\frac{{dt}}{t}} } } \right\}^{{1 \mathord{\left/
 {\vphantom {1 2}} \right.
 \kern-\nulldelimiterspace} 2}}}\\
    &= C{\left\{ {\sum\limits_{k \in \mathbb{Z}} {\int_1^2 {{{\left[ {M_{a,L}^ * \left( f \right)\left( {x,{2^{-k}}t} \right)} \right]}^2}\frac{{dt}}{t}} } } \right\}^{{1 \mathord{\left/
 {\vphantom {1 2}} \right.
 \kern-\nulldelimiterspace} 2}}}.
\end{align*}
In the view of Lemma \ref{Lemma_3_1}, we see that for any $\beta > 0, r > 0$ and $a > m / 2$, there exists a constant $C$ such that for all $f \in L^{2}\left(X\right), k \in \mathbb{Z}, x \in X$ and $t \in \left[1, 2\right)$,
}
{
 \setlength\abovedisplayskip{1ex}  
 \setlength\belowdisplayskip{1ex}
\begin{equation}\label{Eq.Lemma_3_1}
{\left| {M_{a,L}^*\left( f \right)\left( {x,{2^{ - k}}t} \right)} \right|^r} \le C\sum\limits_{j = k}^\infty  {{2^{ - \left( {j - k} \right)\beta r}}\int_X {\frac{{|{{({2^{ - j}}t)}^2}L{e^{ - {{({2^{ - j}}t)}^2}L}}f(z){|^r}}}{{V(z,{2^{ - k}}){{[1 + {2^k}d(x,z)]}^{ar}}}}d\mu \left( z \right)} }.
\end{equation}
Let $r \in \left(0, 1\right)$ be as in Corollary \ref{Corollary_Property_of_Growth_Functions_04}, with $p = {2 \mathord{\left/
 {\vphantom {2 r}} \right.
 \kern-\nulldelimiterspace} r} > 1$. Fix some $\beta > 0$ and choose $a > {m \mathord{\left/
 {\vphantom {m 2}} \right.
 \kern-\nulldelimiterspace} 2}$ such that $ar > m + n$. We then take the norm ${\left[ {\int_1^2 {{{\left|  \cdot  \right|}^{{2 \mathord{\left/
 {\vphantom {2 r}} \right.
 \kern-\nulldelimiterspace} r}}}\frac{{dt}}{t}} } \right]^{{r \mathord{\left/
 {\vphantom {r 2}} \right.
 \kern-\nulldelimiterspace} 2}}}$
} on the both sides of \eqref{Eq.Lemma_3_1}, and employ the Minkowski's inequality to get
{
 \setlength\abovedisplayskip{1ex}  
 \setlength\belowdisplayskip{1ex}
\begin{align}
&{\left[ {\int_1^2 {{{\left| {M_{a,L}^*\left( f \right)\left( {x,{2^{ - k}}t} \right)} \right|}^2}\frac{{dt}}{t}} } \right]^{{r \mathord{\left/
 {\vphantom {r 2}} \right.
 \kern-\nulldelimiterspace} 2}}}\nonumber\\&
 \begin{array}{*{20}{c}}
{}&{}
\end{array}\le C{\left\{ {\int_1^2 {{{\left| {\sum\limits_{j = k}^\infty  {{2^{ - \left( {j - k} \right)\beta r}}\int_X {\frac{{|{{({2^{ - j}}t)}^2}L{e^{ - {{({2^{ - j}}t)}^2}L}}f(z){|^r}}}{{V(z,{2^{ - k}}){{[1 + {2^k}d(x,z)]}^{ar}}}}d\mu \left( z \right)} } } \right|}^{{2 \mathord{\left/
 {\vphantom {2 r}} \right.
 \kern-\nulldelimiterspace} r}}}\frac{{dt}}{t}} } \right\}^{{r \mathord{\left/
 {\vphantom {r 2}} \right.
 \kern-\nulldelimiterspace} 2}}}\nonumber\\&
 \begin{array}{*{20}{c}}
{}&{}
\end{array} = C\sum\limits_{j = k}^\infty  {{2^{ - \left( {j - k} \right)\beta r}}{{\left\{ {\int_1^2 {{{\left| {\int_X {\frac{{|{{({2^{ - j}}t)}^2}L{e^{ - {{({2^{ - j}}t)}^2}L}}f(z){|^r}}}{{V(z,{2^{ - k}}){{[1 + {2^k}d(x,z)]}^{ar}}}}d\mu \left( z \right)} } \right|}^{{2 \mathord{\left/
 {\vphantom {2 r}} \right.
 \kern-\nulldelimiterspace} r}}}\frac{{dt}}{t}} } \right\}}^{{r \mathord{\left/
 {\vphantom {r 2}} \right.
 \kern-\nulldelimiterspace} 2}}}} \nonumber\\&
 \begin{array}{*{20}{c}}
{}&{}
\end{array} \le C\sum\limits_{j = k}^\infty  {{2^{ - \left( {j - k} \right)\beta r}}\int_X {\frac{{{{[\int_1^2 {|{{({2^{ - j}}t)}^2}L{e^{ - {{({2^{ - j}}t)}^2}L}}f(z){|^2}{{dt} \mathord{\left/
 {\vphantom {{dt} t}} \right.
 \kern-\nulldelimiterspace} t}} ]}^{{r \mathord{\left/
 {\vphantom {r 2}} \right.
 \kern-\nulldelimiterspace} 2}}}}}{{V(z,{2^{ - k}}){{[1 + {2^k}d(x,z)]}^{ar}}}}d\mu \left( z \right)} } \nonumber\\&
 \begin{array}{*{20}{c}}
{}&{}
\end{array} = C\int_X {\frac{{\sum\nolimits_{j = k}^\infty  {{2^{ - \left( {j - k} \right)\beta r}}{{[\int_1^2 {|{{({2^{ - j}}t)}^2}L{e^{ - {{({2^{ - j}}t)}^2}L}}f(z){|^2}{{dt} \mathord{\left/
 {\vphantom {{dt} t}} \right.
 \kern-\nulldelimiterspace} t}} ]}^{{r \mathord{\left/
 {\vphantom {r 2}} \right.
 \kern-\nulldelimiterspace} 2}}}} }}{{V(z,{2^{ - k}}){{[1 + {2^k}d(x,z)]}^{ar}}}}d\mu \left( z \right)} \nonumber\\&
 \begin{array}{*{20}{c}}
{}&{}
\end{array} \le C\mathcal{M}\left\{ {\sum\nolimits_{j = k}^\infty  {{2^{ - \left( {j - k} \right)\beta r}}{{\left[\int_1^2 {|{{({2^{ - j}}t)}^2}L{e^{ - {{({2^{ - j}}t)}^2}L}}f( \cdot ){|^2}{{dt} \mathord{\left/
 {\vphantom {{dt} t}} \right.
 \kern-\nulldelimiterspace} t}} \right]}^{{r \mathord{\left/
 {\vphantom {r 2}} \right.
 \kern-\nulldelimiterspace} 2}}}} } \right\}\left( x \right) \nonumber \\&
 \begin{array}{*{20}{c}}
{}&{}
\end{array} : = C\mathcal{M}\left( {{F_k}} \right)\left( x \right),\nonumber
\end{align}
where ${F_k}\left( x \right): = \sum\nolimits_{j = k}^\infty  {{2^{ - \left( {j - k} \right)\beta r}}{{[\int_1^2 {|{{({2^{ - j}}t)}^2}L{e^{ - {{({2^{ - j}}t)}^2}L}}f(x){|^2}{{dt} \mathord{\left/
 {\vphantom {{dt} t}} \right.
 \kern-\nulldelimiterspace} t}} ]}^{{r \mathord{\left/
 {\vphantom {r 2}} \right.
 \kern-\nulldelimiterspace} 2}}}}$, and the last inequality follows from Lemma \ref{Lemma_3_3}. It follows that
}
{
 \setlength\abovedisplayskip{1ex}  
 \setlength\belowdisplayskip{1ex}
\begin{eqnarray}\label{Proof_of_Theorem_3_1_02}
&&\int_X {\varphi \left( {x,\frac{{\left| {{W_f}} \right|}}{\lambda }} \right)d\mu \left( x \right)} \nonumber\\
 &\le& \int_X {\varphi \left( {x,{{C{{\left( {\sum\limits_{k \in \mathbb{Z}} {\int_1^2 {{{\left[ {M_{a,L}^*\left( f \right)\left( {x,{2^{ - k}}t} \right)} \right]}^2}{{dt} \mathord{\left/
 {\vphantom {{dt} t}} \right.
 \kern-\nulldelimiterspace} t}} } } \right)}^{{1 \mathord{\left/
 {\vphantom {1 2}} \right.
 \kern-\nulldelimiterspace} 2}}}} \mathord{\left/
 {\vphantom {{C{{\left( {\sum\limits_{k \in Z} {\int_1^2 {{{\left[ {M_{a,L}^*\left( f \right)\left( {x,{2^{ - k}}t} \right)} \right]}^2}{{dt} \mathord{\left/
 {\vphantom {{dt} t}} \right.
 \kern-\nulldelimiterspace} t}} } } \right)}^{{1 \mathord{\left/
 {\vphantom {1 2}} \right.
 \kern-\nulldelimiterspace} 2}}}} \lambda }} \right.
 \kern-\nulldelimiterspace} \lambda }} \right)d\mu \left( x \right)} \nonumber\\
 &=& \int_X {\varphi \left( {x,{{C{{\left( {\sum\limits_{k \in \mathbb{Z}} {{{\left[ {{{\left( {\int_1^2 {{{\left[ {M_{a,L}^*\left( f \right)\left( {x,{2^{ - k}}t} \right)} \right]}^2}{{dt} \mathord{\left/
 {\vphantom {{dt} t}} \right.
 \kern-\nulldelimiterspace} t}} } \right)}^{{r \mathord{\left/
 {\vphantom {r 2}} \right.
 \kern-\nulldelimiterspace} 2}}}} \right]}^{{2 \mathord{\left/
 {\vphantom {2 r}} \right.
 \kern-\nulldelimiterspace} r}}}} } \right)}^{{1 \mathord{\left/
 {\vphantom {1 2}} \right.
 \kern-\nulldelimiterspace} 2}}}} \mathord{\left/
 {\vphantom {{C{{\left( {\sum\limits_{k \in \mathbb{Z}} {{{\left[ {{{\left( {\int_1^2 {{{\left[ {M_{a,L}^*\left( f \right)\left( {x,{2^{ - k}}t} \right)} \right]}^2}{{dt} \mathord{\left/
 {\vphantom {{dt} t}} \right.
 \kern-\nulldelimiterspace} t}} } \right)}^{{r \mathord{\left/
 {\vphantom {r 2}} \right.
 \kern-\nulldelimiterspace} 2}}}} \right]}^{{2 \mathord{\left/
 {\vphantom {2 r}} \right.
 \kern-\nulldelimiterspace} r}}}} } \right)}^{{1 \mathord{\left/
 {\vphantom {1 2}} \right.
 \kern-\nulldelimiterspace} 2}}}} \lambda }} \right.
 \kern-\nulldelimiterspace} \lambda }} \right)d\mu \left( x \right)} \nonumber\\
 &\le& \int_X {\varphi \left( {x,{{C{{\left( {\sum\limits_{k \in \mathbb{Z}} {{{\left[ {\mathcal{M}\left( {{F_k}} \right)\left( x \right)} \right]}^{{2 \mathord{\left/
 {\vphantom {2 r}} \right.
 \kern-\nulldelimiterspace} r}}}} } \right)}^{{1 \mathord{\left/
 {\vphantom {1 2}} \right.
 \kern-\nulldelimiterspace} 2}}}} \mathord{\left/
 {\vphantom {{C{{\left( {\sum\limits_{k \in \mathbb{Z}} {{{\left[ {M\left( {{F_k}} \right)\left( x \right)} \right]}^{{2 \mathord{\left/
 {\vphantom {2 r}} \right.
 \kern-\nulldelimiterspace} r}}}} } \right)}^{{1 \mathord{\left/
 {\vphantom {1 2}} \right.
 \kern-\nulldelimiterspace} 2}}}} \lambda }} \right.
 \kern-\nulldelimiterspace} \lambda }} \right)d\mu \left( x \right)} \nonumber\\
 &\le& C\int_X {\varphi \left( {x,{{{{\left( {\sum\limits_{k \in \mathbb{Z}} {{F_k}{{\left( x \right)}^{{2 \mathord{\left/
 {\vphantom {2 r}} \right.
 \kern-\nulldelimiterspace} r}}}} } \right)}^{{1 \mathord{\left/
 {\vphantom {1 2}} \right.
 \kern-\nulldelimiterspace} 2}}}} \mathord{\left/
 {\vphantom {{{{\left( {\sum\limits_{k \in Z} {{F_k}{{\left( x \right)}^{{2 \mathord{\left/
 {\vphantom {2 r}} \right.
 \kern-\nulldelimiterspace} r}}}} } \right)}^{{1 \mathord{\left/
 {\vphantom {1 2}} \right.
 \kern-\nulldelimiterspace} 2}}}} \lambda }} \right.
 \kern-\nulldelimiterspace} \lambda }} \right)d\mu \left( x \right)},
\end{eqnarray}
where we used the Corollary \ref{Corollary_Property_of_Growth_Functions_04} in the last inequality.
}

We now turn to estimate ${F_k}{\left( x \right)^{{2 \mathord{\left/
 {\vphantom {2 r}} \right.
 \kern-\nulldelimiterspace} r}}}$. For any $k \in \mathbb{Z}$, we recall
 $$
 \setlength\abovedisplayskip{1ex}  
 \setlength\belowdisplayskip{1ex}
{F_k}\left( x \right) = \sum\nolimits_{j = k}^\infty  {{2^{ - \left( {j - k} \right)\beta r}}{{\left[ {\int_1^2 {|{{({2^{ - j}}t)}^2}L{e^{ - {{({2^{ - j}}t)}^2}L}}f(x){|^2}{{dt} \mathord{\left/
 {\vphantom {{dt} t}} \right.
 \kern-\nulldelimiterspace} t}} } \right]}^{{r \mathord{\left/
 {\vphantom {r 2}} \right.
 \kern-\nulldelimiterspace} 2}}}}.
 $$
 Since
 $$
 \setlength\abovedisplayskip{1ex}  
 \setlength\belowdisplayskip{1ex}
 {2^{ - \left( {j - k} \right)\beta r}} = \frac{{\beta r}}{{1 - {2^{ - \beta r}}}}\int_{{2^{j - k}}}^{{2^{j - k + 1}}} {{s^{ - \beta r - 1}}ds},
 $$we let $E_{j}:= {{\left[ {{2^{j - k}},{2^{j - k + 1}}} \right]}}$ and it follows that
 {
 \setlength\abovedisplayskip{1ex}  
 \setlength\belowdisplayskip{1ex}
 \begin{align}
{F_k}\left( x \right)
&=C\sum\nolimits_{j = k}^\infty  {\int_{{2^{j - k}}}^{{2^{j - k + 1}}} {{{\left[ {\int_1^2 {{\rm{|}}{{({2^{ - j}}t)}^2}L{e^{ - {{({2^{ - j}}t)}^2}L}}f\left( x \right){{\rm{|}}^2}\frac{{dt}}{t}} } \right]}^{{r \mathord{\left/
 {\vphantom {r 2}} \right.
 \kern-\nulldelimiterspace} 2}}}\frac{{ds}}{{{s^{\beta r + 1}}}}} } \nonumber\\
&=C\sum\nolimits_{j = k}^\infty  {\int_1^\infty  {{{\left[ {\int_1^2 {{\rm{|}}{{({2^{ - j}}t)}^2}L{e^{ - {{({2^{ - j}}t)}^2}L}}f\left( x \right){{\rm{|}}^2}\frac{{dt}}{t}} } \right]}^{{r \mathord{\left/
 {\vphantom {r 2}} \right.
 \kern-\nulldelimiterspace} 2}}}{\mathcal{X}_{{E_j}}}\left( s \right)\frac{{ds}}{{{s^{\beta r + 1}}}}} } \nonumber\\
&=C\int_1^\infty  {\sum\nolimits_{j = k}^\infty  {{{\left[ {\int_1^2 {{\rm{|}}{{({2^{ - j}}t)}^2}L{e^{ - {{({2^{ - j}}t)}^2}L}}f\left( x \right){{\rm{|}}^2}\frac{{dt}}{t}} } \right]}^{{r \mathord{\left/
 {\vphantom {r 2}} \right.
 \kern-\nulldelimiterspace} 2}}}{\mathcal{X}_{{E_j}}}\left( s \right)\frac{{ds}}{{{s^{\beta r + 1}}}}} }. \nonumber
 \end{align}
 We then apply the H\"{o}lder's inequality to obtain
 }
 {
 \setlength\abovedisplayskip{1ex}  
 \setlength\belowdisplayskip{1ex}
 \begin{align}\label{Proof_of_Theorem_3_1_06}
{F_k}{\left( x \right)^{{2 \mathord{\left/
 {\vphantom {2 r}} \right.
 \kern-\nulldelimiterspace} r}}}
 &= {\left[ {C\int_1^\infty  {\sum\nolimits_{j = k}^\infty  {{{\left[ {\int_1^2 {|{{({2^{ - j}}t)}^2}L{e^{ - {{({2^{ - j}}t)}^2}L}}f\left( x \right){|^2}\frac{{dt}}{t}} } \right]}^{{r \mathord{\left/
 {\vphantom {r 2}} \right.
 \kern-\nulldelimiterspace} 2}}}{{\cal X}_{{E_j}}}\left( s \right)\frac{{ds}}{{{s^{\beta r + 1}}}}} } } \right]^{{2 \mathord{\left/
 {\vphantom {2 r}} \right.
 \kern-\nulldelimiterspace} r}}}\nonumber\\
 &\le C{\left( {\int_1^\infty  {{s^{ - \beta r - 1}}ds} } \right)^{{4 \mathord{\left/
 {\vphantom {4 {r\left( {2 - r} \right)}}} \right.
 \kern-\nulldelimiterspace} {r\left( {2 - r} \right)}}}}\nonumber\\&
 \begin{array}{*{20}{c}}
{}&{}
\end{array}\times \int_1^\infty  {{{\left( {\sum\nolimits_{j = k}^\infty  {{{\left[ {\int_1^2 {|{{({2^{ - j}}t)}^2}L{e^{ - {{({2^{ - j}}t)}^2}L}}f\left( x \right){|^2}\frac{{dt}}{t}} } \right]}^{{r \mathord{\left/
 {\vphantom {r 2}} \right.
 \kern-\nulldelimiterspace} 2}}}{{\cal X}_{{E_j}}}\left( s \right)} } \right)}^{{2 \mathord{\left/
 {\vphantom {2 r}} \right.
 \kern-\nulldelimiterspace} r}}}\frac{{ds}}{{{s^{\beta r + 1}}}}} \nonumber\\
 &= C\int_1^\infty  {\left( {\sum\nolimits_{j = k}^\infty  {\left( {\int_1^2 {|{{({2^{ - j}}t)}^2}L{e^{ - {{({2^{ - j}}t)}^2}L}}f\left( x \right){|^2}\frac{{dt}}{t}} } \right){{\cal X}_{{E_j}}}\left( s \right)} } \right)\frac{{ds}}{{{s^{\beta r + 1}}}}}\nonumber\\
 &= C\sum\nolimits_{j = k}^\infty  {\left( {\int_{{2^{j - k}}}^{{2^{j - k + 1}}} {\frac{{ds}}{{{s^{\beta r + 1}}}}\int_1^2 {|{{({2^{ - j}}t)}^2}L{e^{ - {{({2^{ - j}}t)}^2}L}}f\left( x \right){|^2}\frac{{dt}}{t}} } } \right)} \nonumber\\
 &= C\sum\nolimits_{j = k}^\infty  {\left( {{2^{ - \left( {j - k} \right)\beta r}}\int_1^2 {|{{({2^{ - j}}t)}^2}L{e^{ - {{({2^{ - j}}t)}^2}L}}f\left( x \right){|^2}\frac{{dt}}{t}} } \right)}.
 \end{align}
 Summation by all $k \in \mathbb{Z}$, we have
 }
 {
 \setlength\abovedisplayskip{1ex}  
 \setlength\belowdisplayskip{1ex}
 \begin{align*}
\sum\limits_{k \in \mathbb{Z}} {{F_k}{{\left( x \right)}^{{2 \mathord{\left/
 {\vphantom {2 r}} \right.
 \kern-\nulldelimiterspace} r}}}}
 &\le C\sum\limits_{k \in \mathbb{Z}} {\sum\limits_{j \ge k} {\left( {{2^{ - \left( {j - k} \right)\beta r}}\int_1^2 {|{{({2^{ - j}}t)}^2}L{e^{ - {{({2^{ - j}}t)}^2}L}}f\left( x \right){|^2}\frac{{dt}}{t}} } \right)} } \\
 &= C\sum\limits_{j \in \mathbb{Z}} {\sum\limits_{k \le j} {\left( {{2^{ - \left( {j - k} \right)\beta r}}\int_1^2 {|{{({2^{ - j}}t)}^2}L{e^{ - {{({2^{ - j}}t)}^2}L}}f\left( x \right){|^2}\frac{{dt}}{t}} } \right)} } \\
 &= C{(1 - {2^{ - \beta r}})^{ - 1}}\sum\limits_{j \in Z} {\int_1^2 {|{{({2^{ - j}}t)}^2}L{e^{ - {{({2^{ - j}}t)}^2}L}}f\left( x \right){|^2}\frac{{dt}}{t}} } \\
 &= C\sum\limits_{j \in \mathbb{Z}} {\int_{{2^{ - j}}}^{{2^{ - j + 1}}} {|{t^2}L{e^{ - {t^2}L}}f\left( x \right){|^2}\frac{{dt}}{t}} } \\
 &= C\int_0^\infty  {|{t^2}L{e^{ - {t^2}L}}f\left( x \right){|^2}\frac{{dt}}{t}} \\
 &= C{G_L}\left( f \right){\left( x \right)^2},
 \end{align*}
 which, together with \eqref{Proof_of_Theorem_3_1_02} and \eqref{Eq.Property_of_Growth_Function}, yields the $\ge$ inequality in \eqref{Proof_of_Theorem_3_1_01}.
 }

 It remains to establish the reverse inequality. In the view of Proposition \ref{Pro.ATLM_Family}, we write
 $f = \sum\limits_{k \in \mathbb{Z}} {\sum\limits_{\alpha  \in {I_k}} {{s_{Q_\alpha ^k}}{a_{Q_\alpha ^k}}} }$. Let $\delta$ be as in Lemma \ref{Lemma_Decompositon_of_Space} we get
 {
 \setlength\abovedisplayskip{1ex}  
 \setlength\belowdisplayskip{1ex}
 \begin{align}\label{Proof_of_Theorem_3_1_03}
{G_L}\left( f \right)\left( x \right) =& {\left( {\int_0^\infty  {|{t^2}L{e^{ - {t^2}L}}\left( f \right)\left( x \right){|^2}\frac{{dt}}{t}} } \right)^{{1 \mathord{\left/
 {\vphantom {1 2}} \right.
 \kern-\nulldelimiterspace} 2}}}\nonumber\\
 =& {\left( {\int_0^\infty  {|{t^2}L{e^{ - {t^2}L}}(\sum\limits_{k \in \mathbb{Z}} {\sum\limits_{\alpha  \in {I_k}} {{s_{Q_\alpha ^k}}{a_{Q_\alpha ^k}}} } )\left( x \right){|^2}\frac{{dt}}{t}} } \right)^{{1 \mathord{\left/
 {\vphantom {1 2}} \right.
 \kern-\nulldelimiterspace} 2}}}\nonumber\\
 =& {\left( {\sum\limits_{j \in \mathbb{Z}} {\int_{{\delta^{ j}}}^{{\delta^{ j - 1}}} {|{t^2}L{e^{ - {t^2}L}}(\sum\limits_{k \in \mathbb{Z}} {\sum\limits_{\alpha  \in {I_k}} {{s_{Q_\alpha ^k}}{a_{Q_\alpha ^k}}} } )\left( x \right){|^2}\frac{{dt}}{t}} } } \right)^{{1 \mathord{\left/
 {\vphantom {1 2}} \right.
 \kern-\nulldelimiterspace} 2}}}\nonumber\\
 \le& {\left( {\sum\limits_{j \in \mathbb{Z}} {\int_{{\delta^{ j}}}^{{\delta^{  j - 1}}} {|{t^2}L{e^{ - {t^2}L}}(\sum\limits_{k > j} {\sum\limits_{\alpha  \in {I_k}} {{s_{Q_\alpha ^k}}{a_{Q_\alpha ^k}}} } )\left( x \right){|^2}\frac{{dt}}{t}} } } \right)^{{1 \mathord{\left/
 {\vphantom {1 2}} \right.
 \kern-\nulldelimiterspace} 2}}}\nonumber\\&
\begin{array}{*{20}{c}}
{}&{}&{}&{}
\end{array} + {\left( {\sum\limits_{j \in \mathbb{Z}} {\int_{{\delta^{  j}}}^{{\delta^{  j - 1}}} {|{t^2}L{e^{ - {t^2}L}}(\sum\limits_{k \le j} {\sum\limits_{\alpha  \in {I_k}} {{s_{Q_\alpha ^k}}{a_{Q_\alpha ^k}}} } )\left( x \right){|^2}\frac{{dt}}{t}} } } \right)^{{1 \mathord{\left/
 {\vphantom {1 2}} \right.
 \kern-\nulldelimiterspace} 2}}}.
 \end{align}

 We now estimate the first part of \eqref{Proof_of_Theorem_3_1_03}. For any $k > j$ and $\alpha \in I_k$, noting that $a_{Q_{\alpha}^k} = L^Mb_{Q_{\alpha}^k}$, we write
 }
 $$
 \setlength\abovedisplayskip{1ex}  
 \setlength\belowdisplayskip{1ex}
 \left| {{t^2}L{e^{ - {t^2}L}}\left( {{a_{Q_\alpha ^k}}} \right)(x)} \right| = \left| {{t^2}{L^{M + 1}}{e^{ - {t^2}L}}\left( {{b_{Q_\alpha ^k}}} \right)(x)} \right| = {t^{ - 2M}}\left| {{{\left( {{t^2}L} \right)}^{M + 1}}{e^{ - {t^2}L}}\left( {{b_{Q_\alpha ^k}}} \right)(x)} \right|.
 $$

 Let $n$ be as in \eqref{Eq.Doubling_dim_of_X_01}, since $M > {{nq\left( \varphi  \right)} \mathord{\left/
 {\vphantom {{nq\left( \varphi  \right)} {\left( {2{p_1}} \right)}}} \right.
 \kern-\nulldelimiterspace} {\left( {2{p_1}} \right)}}$, we can choose some $q = r$ with $r$ be as in Corollary \ref{Corollary_Property_of_Growth_Functions_04} such that $2M > {n \mathord{\left/
 {\vphantom {n q}} \right.
 \kern-\nulldelimiterspace} q}$. We then let $N$ be some positive number such that $2M > N > {n \mathord{\left/
 {\vphantom {n q}} \right.
 \kern-\nulldelimiterspace} q}$. Then by Definition \ref{Def.AT_LM_Family}, the upper bound of the kernel of $\left(t^2L\right)^{M+1}e^{-t^2L}$ and \eqref{Eq.Doubling_Int_Properties}, we get
 {
 \setlength\abovedisplayskip{1ex}  
 \setlength\belowdisplayskip{1ex}
 \begin{align}
&\left| {{t^2}L{e^{ - {t^2}L}}\left( {{a_{Q_\alpha ^k}}} \right)(x)} \right|\nonumber\\&
\begin{array}{*{20}{c}}
{}&{}&{}
\end{array} \le \frac{{{C_5}}}{{V\left( {x,t} \right)}}{t^{ - 2M}}\ell {\left( {Q_\alpha ^k} \right)^{2M}}\mu {\left( {Q_\alpha ^k} \right)^{{{ - 1} \mathord{\left/
 {\vphantom {{ - 1} 2}} \right.
 \kern-\nulldelimiterspace} 2}}}\int_{3Q_\alpha ^k} {\exp \left( { - \frac{{d{{\left( {x,y} \right)}^2}}}{{{C_6}{t^2}}}} \right)} d\mu \left( y \right)\nonumber\\&
\begin{array}{*{20}{c}}
{}&{}&{}
\end{array} \le C{t^{ - 2M}}\ell {\left( {Q_\alpha ^k} \right)^{2M}}\mu {\left( {Q_\alpha ^k} \right)^{{{ - 1} \mathord{\left/
 {\vphantom {{ - 1} 2}} \right.
 \kern-\nulldelimiterspace} 2}}}{\left( {\frac{t}{{t + d(x,y_\alpha ^k)}}} \right)^N},\nonumber
 \end{align}
 where we denote by $y_{\alpha}^{k}$ the center of $Q_{\alpha}^k$. Hence,
 }
 {
 \setlength\abovedisplayskip{1ex}  
 \setlength\belowdisplayskip{1ex}
 \begin{align}\label{Proof_of_Theorem_3_1_04}
&|{t^2}L{e^{ - {t^2}L}}(\sum\limits_{k > j} {\sum\limits_{\alpha  \in {I_k}} {{s_{Q_\alpha ^k}}{a_{Q_\alpha ^k}}} } )\left( x \right)|\nonumber\\&
\begin{array}{*{20}{c}}
{}&{}&{}
\end{array} \le C\sum\limits_{k > j} {\sum\limits_{\alpha  \in {I_k}} {{t^{ - 2M}}\ell {{\left( {Q_\alpha ^k} \right)}^{2M}}\mu {{\left( {Q_\alpha ^k} \right)}^{{{ - 1} \mathord{\left/
 {\vphantom {{ - 1} 2}} \right.
 \kern-\nulldelimiterspace} 2}}}\frac{{|{s_{Q_\alpha ^k}}|}}{{{{\left[ {1 + {t^{ - 1}}d(x,{y_\alpha ^k})} \right]}^N}}}} } \nonumber\\&
\begin{array}{*{20}{c}}
{}&{}&{}
\end{array} \le C\sum\limits_{k > j} {{\delta ^{\left( {2M - N} \right)\left( {k - j} \right)}}\sum\limits_{\alpha  \in {I_k}} {\mu {{\left( {Q_\alpha ^k} \right)}^{{{ - 1} \mathord{\left/
 {\vphantom {{ - 1} 2}} \right.
 \kern-\nulldelimiterspace} 2}}}\frac{{|{s_{Q_\alpha ^k}}|}}{{{{\left[ {1 + \ell {{\left( {Q_\alpha ^k} \right)}^{ - 1}}d(x,{y_\alpha ^k})} \right]}^N}}}} } \nonumber\\&
\begin{array}{*{20}{c}}
{}&{}&{}
\end{array} \le C\sum\limits_{k > j} {{\delta ^{\left( {2M - N} \right)\left( {k - j} \right)}}{{\left[ {\mathcal{M}\left( {\sum\limits_{\alpha  \in {I_k}} {|{s_{Q_\alpha ^k}}{|^q}\mu {{\left( {Q_\alpha ^k} \right)}^{{{ - q} \mathord{\left/
 {\vphantom {{ - q} 2}} \right.
 \kern-\nulldelimiterspace} 2}}}{\mathcal{X}_{Q_\alpha ^k}}\left(  \cdot  \right)} } \right)\left( x \right)} \right]}^{{1 \mathord{\left/
 {\vphantom {1 q}} \right.
 \kern-\nulldelimiterspace} q}}}},
\end{align}
where the last inequality follows from Lemma \ref{Lemma_3_2}.

Estimate of the second part of \eqref{Proof_of_Theorem_3_1_03}. For any $k \le j$ and $\alpha \in I_{k}$, we write
 }
 $$
 \setlength\abovedisplayskip{1ex}  
 \setlength\belowdisplayskip{1ex}
 \left| {{t^2}L{e^{ - {t^2}L}}\left( {{a_{Q_\alpha ^k}}} \right)(x)} \right| = {t^2}\left| {{e^{ - {t^2}L}}\left( {L\left( {{a_{Q_\alpha ^k}}} \right)} \right)(x)} \right|.
 $$
 Then by Definition \ref{Def.AT_LM_Family}, the Gaussian estimate \eqref{Eq.GE} and inequality \eqref{Eq.Doubling_Int_Properties}, we get
 {
 \setlength\abovedisplayskip{1ex}  
 \setlength\belowdisplayskip{1ex}
 \begin{align}
&\left| {{t^2}L{e^{ - {t^2}L}}\left( {{a_{Q_\alpha ^k}}} \right)(x)} \right|\nonumber\\
&\begin{array}{*{20}{c}}
{}&{}&{}
\end{array} \le \frac{{{C_5}}}{{V\left( {x,t} \right)}}{t^2}\ell {\left( {Q_\alpha ^k} \right)^{ - 2}}\mu {\left( {Q_\alpha ^k} \right)^{{{ - 1} \mathord{\left/
 {\vphantom {{ - 1} 2}} \right.
 \kern-\nulldelimiterspace} 2}}}\int_{3Q_\alpha ^k} {\exp \left( { - \frac{{d{{\left( {x,y} \right)}^2}}}{{{C_6}{t^2}}}} \right)d\mu \left( y \right)}.\nonumber\\
&\begin{array}{*{20}{c}}
{}&{}&{}
\end{array} \le C{t^2}\ell {\left( {Q_\alpha ^k} \right)^{ - 2}}\mu {\left( {Q_\alpha ^k} \right)^{{{ - 1} \mathord{\left/
 {\vphantom {{ - 1} 2}} \right.
 \kern-\nulldelimiterspace} 2}}}{\left( {1 + \ell {{(Q_\alpha ^k)}^{ - 1}}d\left( {x,y_\alpha ^k} \right)} \right)^{ - N}},\nonumber
 \end{align}
 which implies that
 }
 {
 \setlength\abovedisplayskip{1ex}  
 \setlength\belowdisplayskip{1ex}
 \begin{align}\label{Proof_of_Theorem_3_1_05}
&|{t^2}L{e^{ - {t^2}L}}(\sum\limits_{k \le j} {\sum\limits_{\alpha  \in {I_k}} {{s_{Q_\alpha ^k}}{a_{Q_\alpha ^k}}} } )\left( x \right)|\nonumber\\&
\begin{array}{*{20}{c}}
{}&{}&{}
\end{array} \le C\sum\limits_{k \le j} {\sum\limits_{\alpha  \in {I_k}} {{t^2}\ell {{\left( {Q_\alpha ^k} \right)}^{ - 2}}\mu {{\left( {Q_\alpha ^k} \right)}^{{{ - 1} \mathord{\left/
 {\vphantom {{ - 1} 2}} \right.
 \kern-\nulldelimiterspace} 2}}}\frac{{|{s_{Q_\alpha ^k}}|}}{{{{\left[ {1 + {{\ell {{(Q_\alpha ^k)}^{ - 1}}}}d(x,{y_\alpha ^k})} \right]}^N}}}} } \nonumber\\&
\begin{array}{*{20}{c}}
{}&{}&{}
\end{array} \le C\sum\limits_{k \le j} {{\delta ^{2\left( {j - k} \right)}}\sum\limits_{\alpha  \in {I_k}} {\mu {{\left( {Q_\alpha ^k} \right)}^{{{ - 1} \mathord{\left/
 {\vphantom {{ - 1} 2}} \right.
 \kern-\nulldelimiterspace} 2}}}\frac{{|{s_{Q_\alpha ^k}}|}}{{{{\left[ {1 + \ell {{\left( {Q_\alpha ^k} \right)}^{ - 1}}d(x,{y_\alpha ^k})} \right]}^N}}}} } \nonumber\\&
\begin{array}{*{20}{c}}
{}&{}&{}
\end{array} \le C\sum\limits_{k \le j} {{{\delta^{2\left(j-k\right)}\left[ {\mathcal{M}\left( {\sum\limits_{\alpha  \in {I_k}} {|{s_{Q_\alpha ^k}}{|^q}\mu {{\left( {Q_\alpha ^k} \right)}^{{{ - q} \mathord{\left/
 {\vphantom {{ - q} 2}} \right.
 \kern-\nulldelimiterspace} 2}}}{\mathcal{X}_{Q_\alpha ^k}}\left(  \cdot  \right)} } \right)\left( x \right)} \right]}^{{1 \mathord{\left/
 {\vphantom {1 q}} \right.
 \kern-\nulldelimiterspace} q}}}}.
 \end{align}

By the same technique we used in \eqref{Proof_of_Theorem_3_1_06}, we combine now \eqref{Proof_of_Theorem_3_1_03}-\eqref{Proof_of_Theorem_3_1_05}, and get the following estimate of $G_L\left(f\right)$,
}\\
{
 \setlength\abovedisplayskip{1ex}  
 \setlength\belowdisplayskip{1ex}
\begin{align*}
{G_L}\left( f \right)\left( x \right) &\le {\left( {\sum\limits_{j \in \mathbb{Z}} {\int_{{\delta ^j}}^{{\delta ^{j - 1}}} {|{t^2}L{e^{ - {t^2}L}}(\sum\limits_{k > j} {\sum\limits_{\alpha  \in {I_k}} {{s_{Q_\alpha ^k}}{a_{Q_\alpha ^k}}} } )\left( x \right){|^2}\frac{{dt}}{t}} } } \right)^{{1 \mathord{\left/
 {\vphantom {1 2}} \right.
 \kern-\nulldelimiterspace} 2}}}\\&
\begin{array}{*{20}{c}}
{}&{}&{}&{}&{}&{}&{}&{}
\end{array} + {\left( {\sum\limits_{j \in \mathbb{Z}} {\int_{{\delta ^j}}^{{\delta ^{j - 1}}} {|{t^2}L{e^{ - {t^2}L}}(\sum\limits_{k \le j} {\sum\limits_{\alpha  \in {I_k}} {{s_{Q_\alpha ^k}}{a_{Q_\alpha ^k}}} } )\left( x \right){|^2}\frac{{dt}}{t}} } } \right)^{{1 \mathord{\left/
 {\vphantom {1 2}} \right.
 \kern-\nulldelimiterspace} 2}}}\\
 &\le C\left( {\sum\limits_{j \in \mathbb{Z}} {\int_{{\delta ^j}}^{{\delta ^{j - 1}}} {|\sum\limits_{k > j} {{\delta ^{\left( {2M - N} \right)\left( {k - j} \right)}}{G_k}{{\left( x \right)}^{{1 \mathord{\left/
 {\vphantom {1 q}} \right.
 \kern-\nulldelimiterspace} q}}}} {|^2}} \frac{{dt}}{t}} } \right.\\
&\begin{array}{*{20}{c}}
{}&{}&{}&{}&{}&{}&{}&{}&{}&{}&{}&{}&{}&{}&{}
\end{array}{\left. { + \sum\limits_{j \in \mathbb{Z}} {\int_{{\delta ^j}}^{{\delta ^{j - 1}}} {|\sum\limits_{k \le j} {{\delta ^{2\left( {j - k} \right)}}{G_k}{{\left( x \right)}^{{1 \mathord{\left/
 {\vphantom {1 q}} \right.
 \kern-\nulldelimiterspace} q}}}} {|^2}} \frac{{dt}}{t}} } \right)^{{1 \mathord{\left/
 {\vphantom {1 2}} \right.
 \kern-\nulldelimiterspace} 2}}}\\
\end{align*}
 \begin{align*} &\le C\left( {\sum\limits_{j \in \mathbb{Z}} {\int_{{\delta ^j}}^{{\delta ^{j - 1}}} {\sum\limits_{k > j} {{\delta ^{\left( {2M - N} \right)\left( {k - j} \right)}}{G_k}{{\left( x \right)}^{{2 \mathord{\left/
 {\vphantom {2 q}} \right.
 \kern-\nulldelimiterspace} q}}}} } \frac{{dt}}{t}} } \right.\\&
\begin{array}{*{20}{c}}
{}&{}&{}&{}&{}&{}&{}&{}&{}&{}&{}&{}&{}&{}&{}
\end{array}{\left. { + \sum\limits_{j \in \mathbb{Z}} {\int_{{\delta ^j}}^{{\delta ^{j - 1}}} {\sum\limits_{k \le j} {{\delta ^{2\left( {j - k} \right)}}{G_k}{{\left( x \right)}^{{2 \mathord{\left/
 {\vphantom {2 q}} \right.
 \kern-\nulldelimiterspace} q}}}} } \frac{{dt}}{t}} } \right)^{{1 \mathord{\left/
 {\vphantom {1 2}} \right.
 \kern-\nulldelimiterspace} 2}}}\\
& = C{\left( {\sum\limits_{k \in \mathbb{Z}} {{G_k}{{\left( x \right)}^{{2 \mathord{\left/
 {\vphantom {2 q}} \right.
 \kern-\nulldelimiterspace} q}}}(\sum\limits_{j > k} {{\delta ^{\left( {2M - N} \right)\left( {k - j} \right)}}}  + \sum\limits_{j \ge k} {{\delta ^{2\left( {j - k} \right)}}} )} } \right)^{{1 \mathord{\left/
 {\vphantom {1 2}} \right.
 \kern-\nulldelimiterspace} 2}}}\\
&\le C{\left( {\sum\limits_{k \in \mathbb{Z}} {{G_k}{{\left( x \right)}^{{2 \mathord{\left/
 {\vphantom {2 q}} \right.
 \kern-\nulldelimiterspace} q}}}} } \right)^{{1 \mathord{\left/
 {\vphantom {1 2}} \right.
 \kern-\nulldelimiterspace} 2}}},
\end{align*}
where ${G_k}\left( x \right) = {\cal M}\left( {\sum\nolimits_{\alpha  \in {I_k}} {|{s_{Q_\alpha ^k}}{|^q}\mu {{\left( {Q_\alpha ^k} \right)}^{{{ - q} \mathord{\left/
 {\vphantom {{ - q} 2}} \right.
 \kern-\nulldelimiterspace} 2}}}{{\cal X}_{Q_\alpha ^k}}\left(  \cdot  \right)} } \right)\left( x \right)$, for $k \in \mathbb{Z}$. Hence, by employing Corollary \ref{Corollary_Property_of_Growth_Functions_04}, we have
}
{
 \setlength\abovedisplayskip{1ex}  
 \setlength\belowdisplayskip{1ex}
\begin{align*}
&\int_X {\varphi \left( {x,{{{G_L}\left( f \right)(x)} \mathord{\left/
 {\vphantom {{{G_L}\left( f \right)(x)} \lambda }} \right.
 \kern-\nulldelimiterspace} \lambda }} \right)d\mu \left( x \right)} \\&
\begin{array}{*{20}{c}}
{}&{}&{}&{}
\end{array} \le C\int_X {\varphi \left( {x,{\lambda ^{ - 1}}{{\left( {\sum\nolimits_{k \in \mathbb{Z}} {{G_k}{{\left( x \right)}^{{2 \mathord{\left/
 {\vphantom {2 q}} \right.
 \kern-\nulldelimiterspace} q}}}} } \right)}^{{1 \mathord{\left/
 {\vphantom {1 2}} \right.
 \kern-\nulldelimiterspace} 2}}}} \right)d\mu \left( x \right)} \\&
\begin{array}{*{20}{c}}
{}&{}&{}&{}
\end{array} \le C\int_X {\varphi \left( {x,{\lambda ^{ - 1}}{{\left( {\sum\nolimits_{k \in \mathbb{Z}} {{{\left( {\sum\nolimits_{\alpha  \in {I_k}} {|{s_{Q_\alpha ^k}}{|^q}\mu {{\left( {Q_\alpha ^k} \right)}^{{{ - q} \mathord{\left/
 {\vphantom {{ - q} 2}} \right.
 \kern-\nulldelimiterspace} 2}}}{{\cal X}_{Q_\alpha ^k}}\left( x \right)} } \right)}^{{2 \mathord{\left/
 {\vphantom {2 q}} \right.
 \kern-\nulldelimiterspace} q}}}} } \right)}^{{1 \mathord{\left/
 {\vphantom {1 2}} \right.
 \kern-\nulldelimiterspace} 2}}}} \right)d\mu \left( x \right)} \\&
\begin{array}{*{20}{c}}
{}&{}&{}&{}
\end{array} = C\int_X {\varphi \left( {x,{\lambda ^{ - 1}}{{\left( {\sum\nolimits_{k \in \mathbb{Z}} {\sum\nolimits_{\alpha  \in {I_k}} {|{s_{Q_\alpha ^k}}{|^2}\mu {{\left( {Q_\alpha ^k} \right)}^{ - 1}}{{\cal X}_{Q_\alpha ^k}}\left( x \right)} } } \right)}^{{1 \mathord{\left/
 {\vphantom {1 2}} \right.
 \kern-\nulldelimiterspace} 2}}}} \right)d\mu \left( x \right)} \\&
\begin{array}{*{20}{c}}
{}&{}&{}&{}
\end{array} = C\int_X {\varphi \left( {x,{{{W_f}\left( x \right)} \mathord{\left/
 {\vphantom {{{W_f}\left( x \right)} \lambda }} \right.
 \kern-\nulldelimiterspace} \lambda }} \right)d\mu \left( x \right)},
\end{align*}
which proves the reverse inequality in \eqref{Proof_of_Theorem_3_1_01} and completes the proof of the theorem.
}

\end{proofoftheorem}

Using Theorem \ref{Theorem_3_2} and Theorem \ref{Theorem_3_1}, we immediately obtain the following result.
\begin{theorem}{}\label{Theorem_3_3}
    Suppose $L$ is an operator that satisfies \rm{(\textbf{H1})} \emph{and} \rm{(\textbf{H2})}. \emph{Let $\varphi$ be a growth function with uniformly lower type $p_1$. Then the spaces $H_{\varphi, L}\left(X\right)$ and $H_{L, G, \varphi}\left(X\right)$ coincide and their norms are equivalent.}
\end{theorem}

\end{document}